\documentclass[12pt]{amsart}
\usepackage{latexsym}
\usepackage{amssymb, amsmath,mathrsfs,amsthm}
\usepackage{geometry}
\usepackage[OT2,T1]{fontenc}
\usepackage{setspace}
\usepackage{mathtools}
\usepackage{float}
\usepackage{enumerate}

\usepackage[pagebackref,hypertexnames=false, colorlinks, citecolor=red, linkcolor=red]{hyperref}

\newtheorem{theorem}{Theorem}[section]
\newtheorem*{theorem*}{Theorem}
\newtheorem{lemma}[theorem]{Lemma}
\newtheorem{definition}[theorem]{Definition}

\usepackage{color}

\DeclareSymbolFont{cyrletters}{OT2}{wncyr}{m}{n}
\DeclareMathSymbol{\Sha}{\mathalpha}{cyrletters}{"58}

\theoremstyle{remark}
\newtheorem{remark}[theorem]{Remark}

\begin{document}
\raggedbottom

\title{Applications of Envelopes}
	\date{\today}
\author{Kelly Bickel$^\dagger$, Pamela Gorkin$^\ddagger$, {\protect \and} Trung Tran}
\thanks{$\dagger$ Research supported in part by National Science Foundation 
DMS grant \#1448846.}
\thanks{$\ddagger$ Research supported in part by Simons Foundation Grant \#243653. }
\address{Kelly Bickel, Department of Mathematics, Bucknell University, 380 Olin Science Building, Lewisburg, PA 17837, USA.}
\email{kelly.bickel@bucknell.edu}
\address{Pamela Gorkin, Department of Mathematics, Bucknell University, 380 Olin Science Building, Lewisburg, PA 17837, USA.}
\email{pgorkin@bucknell.edu}
\address{Trung Tran, Department of Mathematics, Bucknell University, 380 Olin Science Building, Lewisburg, PA 17837, USA.}
\email{tbt004@bucknell.edu }
\begin{abstract} Intuitively, an envelope of a family of curves is a curve that is
tangent to a member of the family at each point. Here we use envelopes of families of circles to 
study objects from matrix theory and hyperbolic geometry. First we explore relationships between
numerical ranges of $2\times 2$ matrices and families of circles to study the elliptical range theorem. Then
we deduce a relationship between envelopes and the boundaries of families of intersecting circles 
and use it to find the boundaries of various families of pseudohyperbolic disks.
\end{abstract}

\maketitle

\section{Introduction}
\subsection*{Overview} Let $\mathcal{F}$ denote the family of curves in the $xy$-plane satisfying $F(x,y,t)=0$, for some function $F$. 
A simple but powerful object in the study of such a family of curves is the associated envelope.
Intuitively, an {\it envelope} of $\mathcal{F}$ is a curve that, at each of its points, is tangent to a member of the family.
For example, the picture below gives a family of ellipses whose envelope is an astroid. 

\begin{figure}[H]
\label{fig1}
	\centering
	\includegraphics[width=4cm]{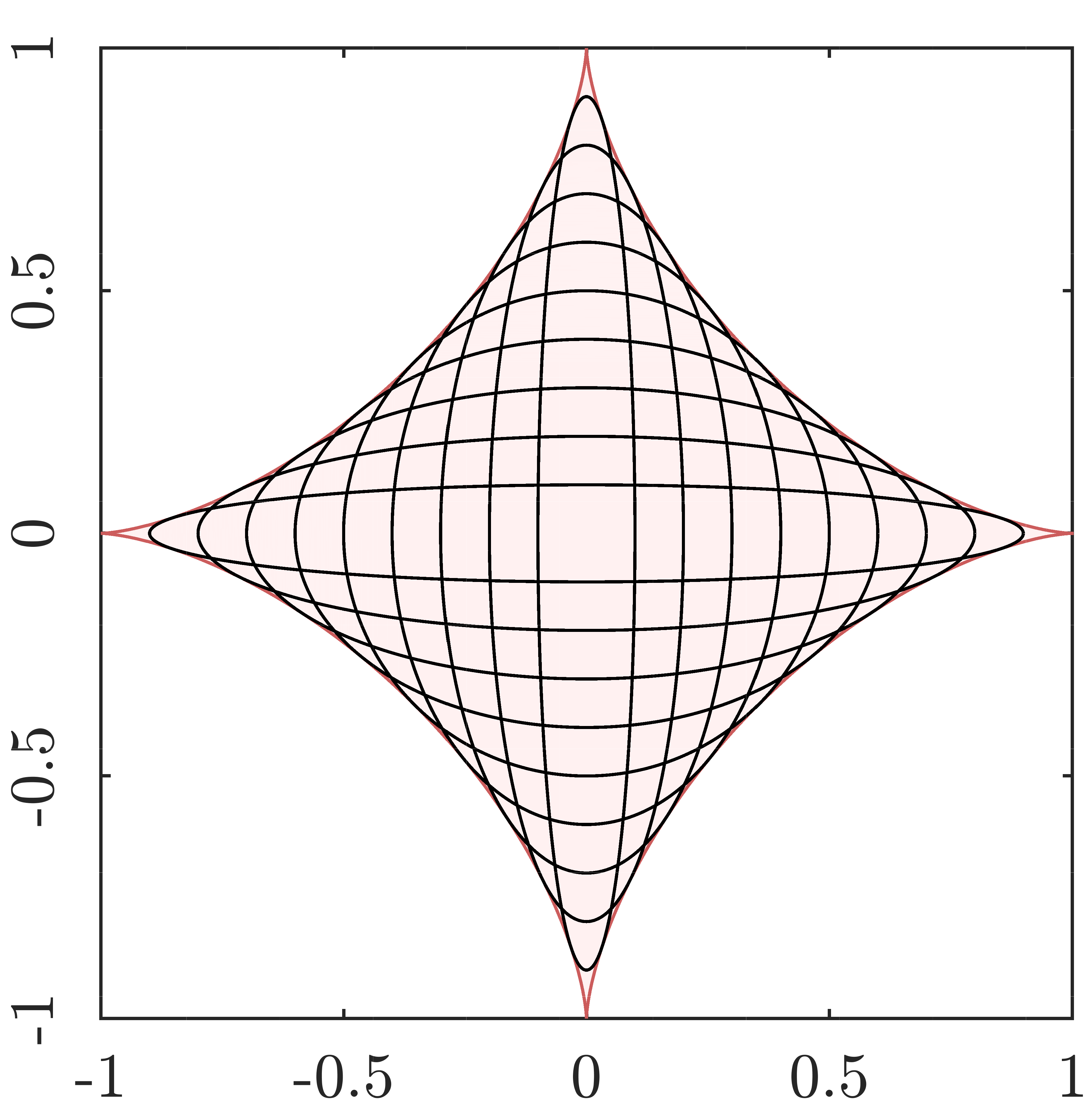}
\caption{\small The ellipses $\tfrac{x^2}{t^2} + \tfrac{y^2}{(t-1)^2}=1$ for $t\in(0,1)$ with envelope $x^{\frac{2}{3}} +y^{\frac{2}{3}}=1$. }
\end{figure}

Envelopes are of practical value and are used in robotics and gear construction, \cite{PP00}. In addition, envelope theorems are used in economics: Economists, including Viner, Harrod, and Schneider, used envelopes to study the connection between short-run and long-run cost curves. For the story of the Viner-Wong theorem, see \cite{M2010}.

In this paper, we illustrate the utility of envelopes. In particular, if one can rewrite a boundary problem in terms of a family of curves, then 
envelope techniques may yield new insights or simpler proofs. Here, we use envelopes to study and provide proofs of both known and new
results from matrix theory and hyperbolic geometry. First, we build off of an argument of Donoghue \cite{don57} to prove the elliptical range
theorem. This proof gives insights about how the numerical range of a $2 \times 2$ matrix comes from an interesting family of circles.

Then we consider general families of intersecting circles and use envelopes to characterize their boundaries. As an application, we study the boundaries
of families of  pseudohyperbolic disks. We first provide a new proof of a well-known result 
about the boundary of a family of pseudohyperbolic disks whose centers lie along a straight line, 
see for example \cite{MR2015, N12}.  This result has numerous applications to function theory on the unit disk. For example, it can be used to analyze Blaschke products whose zeros lie in this family of pseudohyperbolic disks as well as simplify various interpolation problems involving points lying inside this family of disks or on the boundary. For details see Remark \ref{rem:app1}. We then handle a new situation concerning pseudohyperbolic
disks with centers on a horocycle. This result also has applications to interpolation, which are discussed in Remark \ref{rem:app2}.

Before considering these problems further, we need a precise notion of an envelope.

\subsection*{Envelopes} There are several competing ways, to define the envelope of a family $\mathcal{F}$, where $\mathcal{F}$ is a family of curves given by $F(x, y, t) = 0$, for some  continuously differentiable $F$. In what follows, let $\Gamma_{t}$ denote the curve of $(x,y)$ points satisfying $F(x, y, t) = 0$ for $t$ fixed. Perhaps the most common definition is the following:

\begin{definition} A \emph{geometric envelope}  $E_1$  of $\mathcal{F}$ is a curve so that each point on $E_1$ is a point of tangency to some $\Gamma_{t}$ (and often, each $\Gamma_t$ is touched by $E_1$). 
\end{definition}

The geometric envelope is intuitive, but it is difficult to compute directly. Instead, the following definitions are often used in applications:

\begin{definition} The \emph{limiting-position envelope} $E_2$  of $\mathcal{F}$ is the set of points obtained as limits of intersections of nearby $\Gamma_t;$ a point $(x,y) \in E_2$ if there are sequences  $(t_n)$ and  $(\tilde{t}_n)$ converging to some $t$, so that $(x,y)$ is a limit of intersection points of 
$\Gamma_{t_n}$ and $\Gamma_{\tilde{t}_n}$.
\end{definition}

\begin{definition} The \emph{discriminant envelope} $E_3$ of $\mathcal{F}$ is the set of points $(x,y)$ so that there is a value of $t$ with both 
$F(x,y,t ) = 0$ and $F_{t}(x,y, t)=0.$
\end{definition}

In general, these definitions can give different sets of points. 
Nevertheless, it is known that $E_1$ and $E_2$ are both contained in $E_3$; see \cite[Propositions $1$ and $2$]{bruce}. Similarly, if each curve in $E_3$ can be written  as $(x(t), y(t))$ and both 
\[ F_x^2(x,y,t) + F_y^2(x,y,t) \ne 0 \ \ \text{and} \ \ x'(t)^2 +y'(t)^2 \ne 0,\]
 then $E_1 =E_3,$ see \cite[pp. 173]{c88}. One can often find $E_3$ by setting $F(x,y,t) = 0$ and  $F_{t}(x,y,t)=0$ and then eliminating $t$; this is called the \emph{envelope algorithm}. For more information about these and additional definitions, see \cite{bruce, c88, k07, rutter}. For a variety of envelope applications and related results, see  \cite{cap02, gil12, km96, ma13, p47, tak17} and the references therein.
 
Often, the boundary (or a piece of the boundary) of $\mathcal{F}$ will correspond to an envelope of $\mathcal{F}$.  For example, in  Figure \ref{fig1} the astroid is exactly the boundary of the family of ellipses and, except at the corners $(\pm 1, 0), (0, \pm 1)$, it agrees with the discriminant envelope. Indeed, \cite{k07} notes that if the boundary can be smoothly parameterized by $t$, then it must belong to the geometric envelope.  However, without apriori knowledge of the boundary, such a condition is difficult to check. The numerous complications that can arise for general $\mathcal{F}$ make a very general theorem unlikely.

\subsection*{Main Results} In this paper, we consider two distinct applications of envelopes.

\subsubsection*{Application 1: Elliptical range theorem}
Given an $n \times n$ complex matrix $B$, the numerical range of $B$ is the set $W(B)$ defined by
\[W(B) = \{\langle Bx, x\rangle: \|x\| = 1, x \in \mathbb{C}^n\}.\] 
While the numerical range includes the spectrum of $B$, it typically encodes additional information. 
 One particularly deep result about numerical ranges, called the Toeplitz-Hausdorff theorem, states that they are always convex. Most proofs of this result (though not all; see, for example, \cite{d71} for a proof that  does not rely on Theorem~\ref{thm:ERT})
 reduce the problem to $2\times 2$ matrices, see \cite[pp.4]{gr97}, and then use the following fact:

\begin{theorem}[The elliptical range theorem] \label{thm:ERT}
	Let $A$ be a $2 \times 2$ matrix with eigenvalues $a$ and $b$. Then the numerical range of $A$ is an elliptical disk with foci at $a$ and $b$ and minor axis of length $\left(tr(A^\star A) - |a|^2 - |b|^2\right)^{1/2}$.
\end{theorem}

C. K. Li gave a simple computational proof of this in his paper \cite{ckli} and other proofs appear in \cite{gr97, hj91, fdm32}. Here, we present a proof that rests on the discriminant envelope.  We first reduce a general $A$ to $T'=\begin{bmatrix} 0 & m \\ 0& 1\end{bmatrix}$, where $m>0$ and show that $W(T')$ can be expressed as a union of circles. Then envelope techniques, paired with additional computations, allow us to characterize the set covered by this family of circles and obtain Theorem \ref{thm:ERT}. 

This proof is based on one from a paper of Donoghue  \cite{don57}, where he does a similar reduction and uses envelopes to identify the boundary of the numerical range. However in general, the envelope need not be the boundary of the region covered by the family of curves. 
Here we provide the details of the envelope algorithm and in Lemma~\ref{lemma:forward}, add the details that show that in this case, this envelope does give the boundary of the union of circles and hence, the numerical range. 

\subsubsection*{Application 2: Families of Intersecting Circles}
Let $\mathcal{F}$ denote a family of circles parameterized by $t \in [s_1, s_2]$. In particular, assume that 
there is a curve of centers $c(t) = (x_c(t), y_c(t))$ for $t\in [s_1, s_2]$ and a curve of radii $r(t)$ for $t\in [s_1, s_2]$ so that $\mathcal{F}$ is the set of curves with $F(x,y,t)=0$ for
\begin{equation} \label{eqn:circles1} F(x,y,t) = (x-x_c(t))^2 +(y-y_c(t))^2 - r(t)^2.\end{equation}
For each fixed $t$, let $\mathcal{C}_t$ be the circle defined by $F(x,y,t)=0$ and let $\mathcal{D}_t$ be the open disk whose boundary is $\mathcal{C}_t$.

If the curves $c$ of centers and $r$ of radii are sufficiently smooth and nearby circles intersect, then there is a nice relationship between
the boundary of $\cup_t \mathcal{D}_t$ and the limiting-position and discriminant envelopes $E_2$ and $E_3.$ In particular, we prove:

\begin{theorem*}[\ref{thm:circle}] Let $\mathcal{F}$ denote a family of circles parameterized by  $t \in [s_1, s_2]$ as in \eqref{eqn:circles1}. Assume $x_c, y_c, r \in C^2([s_1, s_2])$ and 
for $t\in (s_1, s_2)$, $r(t) >0$ and 
\[ r'(t)^2 < x'_c(t)^2 + y_c'(t)^2.\]
Let $\Omega = \bigcup_{t\in[s_1,s_2]} \mathcal{D}_t.$ Then $\partial \Omega \subseteq E_2 \cup \mathcal{C}_{s_1} \cup \mathcal{C}_{s_2} \subseteq E_3 \cup \mathcal{C}_{s_1} \cup \mathcal{C}_{s_2}.$ 
\end{theorem*}

The proof also yields a parameterization of the points from $E_2$ that can contribute to $\partial \Omega$; this is given in Remark \ref{rem:formulas}. As corollaries, we determine the boundaries of various families of pseudohyperbolic disks. Namely, recall that the pseudohyperbolic distance between two points $z,w\in \mathbb{D}$ is
$$
d_\rho(z,w) = \left|\frac{z-w}{1-\bar{w}z}\right|.
$$
For $\beta \in \mathbb{D}$ and $r \in (0,1)$, let $D_\rho (\beta,r) = \{z \in \mathbb{C}:d_\rho(\beta,z)<r\}$ denote the pseudohyperbolic disk with center $\beta$ and radius $r$. Then $D_\rho (\beta,r)$ is also a Euclidean disk with center $c_{\rho}$ and radius $R_{\rho}$ given by
\begin{equation}\label{eqn:cR}
c_{\rho}(\beta) = \frac{(1-r^2)\beta}{1-r^2|\beta|^2} \ \ \text{ and } \ \ R_{\rho}(\beta) = \frac{r(1-|\beta|^2)}{1-r^2|\beta|^2},
\end{equation} 
see \cite[pp.~$2$]{Gar07}. In particular, families of pseudohyperbolic disks can be realized as families of Euclidean circles, which allows us to study them using Theorem \ref{thm:circle}. \\

\noindent \emph{Case $A$: Disks on a Line.} We first fix $r\in (0,1)$ and examine the family of pseudohyperbolic disks $\{ D_\rho(t,r)\}_{t \in [-1,1]}$, whose centers lie along a real line segment. Specifically, 
define the open disks
 \begin{equation} \label{eqn:disks} \mathcal{D}_1=\{z \in \mathbb{C} :\left|z+ \tfrac{1-r^2}{2r}i \right| < R_r \} \ \ \text{ and }  \ \ \mathcal{D}_2=  \{z \in \mathbb{C} :\left |z-\tfrac{1-r^2}{2r}i \right| < R_r \},\end{equation}
 where the radius $R_r=\frac{1+r^2}{2r}$. Let $\overline{\mathcal{D}_j}$ and $\partial \mathcal{D}_j$ denote the closure and boundary of each $\mathcal{D}_j$ and define the circular arcs 
 \begin{equation} \label{eqn:arcs} \mathcal{A}_1= \partial \mathcal{D}_1 \cap \overline{\mathbb{H}^+} \ \ \text{ and } \ \ \mathcal{A}_2= \partial \mathcal{D}_2 \cap \overline{\mathbb{H}^-},\end{equation}
 where $\mathbb{H}^+=\{ z \in \mathbb{C}: \Im(z) >0\}$ and $\mathbb{H}^-$ is defined analogously.
Then using the envelope algorithm and Theorem \ref{thm:circle}, we show:
\begin{theorem*}[\ref{thm:bdy}] The union $\bigcup_{t\in[-1,1]} D_\rho(t,r)$ has boundary $\mathcal{A}_1 \cup \mathcal{A}_2$.  
\end{theorem*}
Mortini and Rupp gave a different proof of this theorem using complex analysis techniques in \cite{MR2015}. No\"el gave a computational proof of this using limits of intersections of the circles in \cite[Appendix A]{N12}. For applications of this result to function theory on the unit disk, see Remark \ref{rem:app1}.

\noindent \emph{Case $2$: Disks on a Horocycle.}  Fix $r\in (0,1), k>0$, and let $H(1,k)$ denote the circle with radius $\frac{k}{k+1}$ that is internally tangent to $\mathbb{T}$ at $z=1.$ It is easy to show that $H(1,k)$ has center $(\frac{1}{k+1},0)$. Equivalently, one can define
\begin{equation} \label{eqn:horo} H(1,k) = \left \{a \in \mathbb{D}: \tfrac{|a-1|^2}{1-|a|^2}=k\right \} \cup \{ 1\},\end{equation}
where the point $z=1$ is appended so that the set is an actual circle. This set $H(1,k)$ is often called a \emph{horocycle} in $\mathbb{D}$.

We consider the complicated family of pseudohyperbolic disks, $\{ D_\rho(\alpha,r)\}_{\alpha \in H(1,k)}$, whose centers lie along  this horocycle. To see our result,
define
the open disks
\begin{equation} \label{eqn:disks2} \mathscr{D}_1= \left\{z \in \mathbb{C} :\left |z- c_1\right| < R_1 \right\} \ \ \text{ and }  \ \ \mathscr{D}_2=  \{z \in \mathbb{C} :\left|z-c_2\right| < R_2 \} ,
\end{equation}
where
\begin{equation} \label{eqn:c1c2}
\begin{split}
c_1 = \tfrac{1-r}{(1-r) + k(1+r)} \ \ &\text{and} \ \ R_1 = 1-c_1 = \tfrac{k(1+r)}{(1-r)+k(1+r)},\\
c_2 = \tfrac{1+r}{(1+r)+k(1-r)} \ \ &\text{and} \ \  R_2 = 1-c_2 = \tfrac{k(1-r)}{(1+r)+k(1-r)}.
\end{split}
\end{equation}
Let $\overline{\mathscr{D}_j}$ and $\partial \mathscr{D}_j$ denote the closure and boundary of each $\mathscr{D}_j$, respectively. Then we prove: 
\begin{theorem*}[\ref{thm:bdy2}] The union $\bigcup_{\alpha\in H(1, k)} D_\rho(\alpha,r)$ has boundary $\partial \mathscr{D}_1 \cup \partial \mathscr{D}_2$.  
\end{theorem*}
As before, the proof rests on Theorem \ref{thm:circle}. However, as this family of circles is much more complicated, one cannot easily apply the envelope algorithm to get the discriminant envelope. Instead we apply Remark \ref{rem:formulas} to obtain the limiting-position envelope and use that envelope instead. Then Remark \ref{rem:app2} discusses applications of Theorem \ref{thm:bdy2} to interpolation.

\subsection*{Outline of Paper} In what follows, we first prove the elliptical range theorem. Section \ref{sec:ERT1} includes the background and setup, while Section \ref{sec:ERT2} involves the heart of the proof, including the use of the envelope algorithm. 
We then consider Theorem \ref{thm:circle}, which characterizes the boundaries of families of intersecting circles, and its applications to pseudohyperbolic disks. Section \ref{sec:circles} includes both the proof of Theorem \ref{thm:circle} and Remark \ref{rem:formulas}, which illuminates the structure of the limiting-position envelope. Sections \ref{sec:line} and \ref{sec:horocycle} detail the proofs of Theorem \ref{thm:bdy} and Theorem \ref{thm:bdy2}, which apply Theorem \ref{thm:circle} to families of pseudohyperbolic disks whose centers lie along a line segment and a horocycle respectively.

\section*{Acknowledgements}
The authors would like to thank Elias Wegert for sharing the envelope algorithm and Phuong Nguyen for her work on pseudohyperbolic disks and in particular, her contributions to this new proof of Theorem \ref{thm:bdy}.

\section{Elliptical range theorem: background and Setup} \label{sec:ERT1}

\subsection{Background}To prove the elliptical range theorem, we will require the following basic properties of numerical ranges, which one can find in the early sections of \cite{K08}.

\begin{theorem}[Elementary properties]\label{basic} Let $B$ be an $n\times n$ matrix. Then
	\begin{enumerate}[a.]
		\item If $U$ is an $n\times n$ unitary matrix, then
		$
		W(U^\star B U)=W(B).
		$
		
		\item If $\alpha, \beta \in \mathbb{C}$, then
		$
		W(\alpha B + \beta I) = \alpha W(B)+\beta := \{\alpha z + \beta: z\in W(B)\}.
		$
		
		\item If $\lambda\in \mathbb{C}$, then $W(B) = \{\lambda\}$ if and only if $B = \lambda I$.		
	\end{enumerate}	
\end{theorem}	

We also need an important result about $2\times 2$ matrices.

\begin{lemma} \label{l1}
	Let $A$ be a $2\times 2$ matrix with eigenvalues $a$ and $b$. Then $A$ is unitarily equivalent to an upper triangular matrix $T$, where $T=\begin{bmatrix}
	a& p\\
	0& b\\
	\end{bmatrix}$, for $p \geq 0$.	
\end{lemma}
 
This lemma is well known and is a direct consequence of Schur's theorem \cite[Section $2.3$]{hj90}. It follows that since $A$ is unitarily equivalent to $T$, then $A^\star A$ is also unitarily equivalent to $T^\star T$ and so,  
$\text{tr}(A^\star A) = \text{tr}(T^\star T) = |a|^2 + |b|^2 +p^2.$ This is equivalent to 
\begin{equation}\label{p}
	p = (\text{tr}(A^\star A) - |a|^2 - |b|^2)^{1/2},
\end{equation}
which is the minor axis in the elliptical range theorem. 
We first examine a simple example.

\begin{lemma}\label{NumOfR}
	For
	$
	R = 
	\begin{bmatrix}
	0&1\\
	0&0
	\end{bmatrix}, 
	$
	$W(R)$ is a circular disk of radius $\frac{1}{2}$ centered at the origin.  
\end{lemma}	
\begin{proof} A vector $z \in \mathbb{C}^2$ satisfies $\| z\|=1$ precisely when we can write
$$
z = \begin{bmatrix}te^{i\theta_1}\\ \sqrt{1-t^2}e^{i\theta_2}\end{bmatrix}, \text{where } \theta_1, \theta_2 \in [0,2\pi] \text{ and } t\in [0, 1].
$$
A simple computation gives $\langle Rz, z \rangle = e^{i\theta} t\sqrt{1-t^2}$, where $\theta = (\theta_2-\theta_1)$. Letting $\theta$ range over $[0,2\pi]$ shows $W(R)$ is composed of the set of circles centered at the origin with radius  $t \sqrt{1-t^2}$, for $t\in [0,1].$ As $t \sqrt{1-t^2}$ attains precisely the values
 between $0$ and $\frac{1}{2}$ on $[0,1]$, the numerical range $W(R)$ is the circular disk centered at the origin with radius $\frac{1}{2}$. 	
\end{proof}

\begin{figure}[H]
\centering
	\includegraphics[width=4cm]{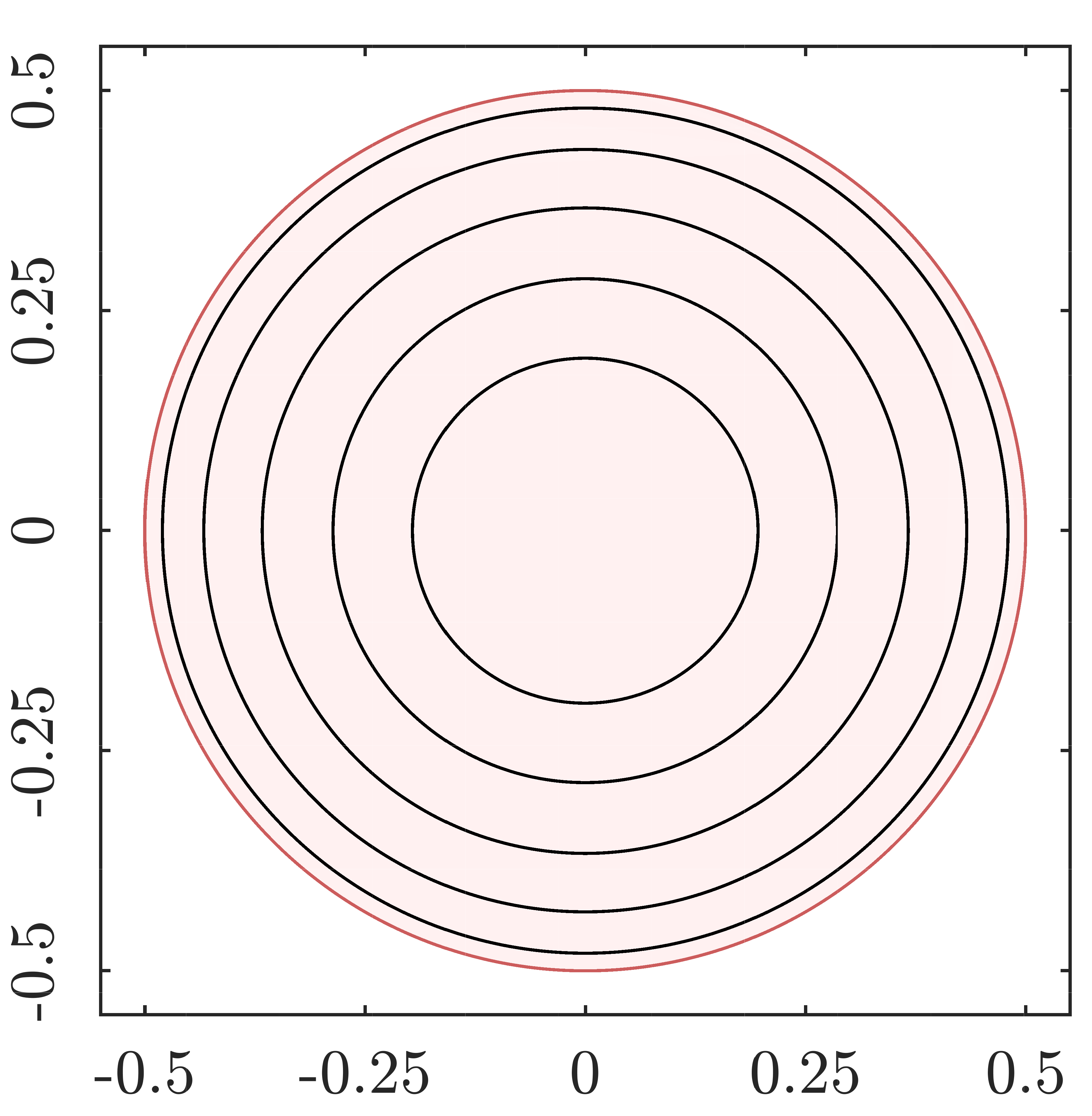}
\caption{\small $W(R)$ as a family of circles.}
\end{figure}

In what follows, $R$ will denote the matrix from Lemma \ref{NumOfR}.

\subsection{Proof Setup} Fix a $2\times 2$ matrix $A$. We will prove the elliptical range theorem by considering two separate cases:  in Subsection \ref{idenCase},
we consider  $A$ with repeated eigenvalues and  in Subsection \ref{distCase}, we consider $A$ with distinct eigenvalues. The proof of the distinct case relies on an additional result proved in Section \ref{sec:ERT2}.

\subsubsection{Repeated eigenvalues} \label{idenCase} Suppose that the eigenvalues of $A$ are equal. Then:
\begin{lemma}
	 Let $A$ be a $2\times 2$ matrix. If $A$ has repeated eigenvalue $a$, then $W(A)$ is a circular disk with center at $a$ and diameter $(\text{tr}(A^\star A) - 2|a|^2)^{1/2}$.	
\end{lemma}

\begin{proof} If $A$ has repeated eigenvalue $a$, then Theorem \ref{basic} implies that 
$$
W(A) = W(T) = W(pR+aI) = pW(R)+a.
$$
Then by Lemma \ref{NumOfR}, $W(R)$ is circular disk of radius $\frac{1}{2}$ centered at the origin.  Addition of $a$ translates the center of the circular disk from the origin to the eigenvalue $a$.  Multiplication of $W(R)$ by $p$ scales the diameter of this circular disk from 1 to $p$. Hence, the diameter of $W(A)$ must equal $p = (\text{tr}(A^\star A) - 2|a|^2)^{1/2}$, by \eqref{p}.
\end{proof}

\subsubsection{Distinct eigenvalues} \label{distCase}
Suppose $A$ has distinct eigenvalues $a$ and $b$ and define
$$
\tilde{T} = \frac{T - aI}{b - a} = \begin{bmatrix}
0&\frac{p}{b-a}\\
0&1
\end{bmatrix} \ \ \text{ and }  \ \ T' = \begin{bmatrix}
0&m\\
0&1
\end{bmatrix} 
$$ 
for $m = \frac{p}{|b-a|} \ge 0$. Then $\tilde{T}$ and $T'$ are unitarily equivalent and by Theorem \ref{basic}, we have
$$
W(A) = W(T) = (b-a)W(\tilde{T})+a =(b-a)W\left (T' \right)+a.
$$
If $m =0$, then Lemma \ref{lem:Ct} will show that $W(T')$ is the line segment $[0,1]$ and so, $W(A)$ is the closed line segment from $a$ to $b$, as needed. If $m>0$, 
let $\overline{\mathcal{E}}$ denote the closed elliptical disk satisfying the inequality
\begin{equation} \label{eqn:ED}
	\frac{\left(x-\frac{1}{2}\right)^2}{1+m^2} +\frac{y^2}{m^2} \leq \frac{1}{4}.
\end{equation}
We denote its interior by $\mathcal{E}$ and its boundary by $\partial \mathcal{E}$. In Section \ref{sec:ERT2}, we will prove
\begin{theorem}\label{thm:E}
	The numerical range $W(T') = \overline{\mathcal{E}}$.
\end{theorem}

The elliptical range theorem for $A$ follows easily from Theorem \ref{thm:E}. First, the relationship between $W(T')$ and $W(A)$ implies that $W(A)$ is also a closed elliptical disk.  By \eqref{eqn:ED}, $W(T')$ has foci $z=0$ and $z=1$ and minor axis of length $m$. Scaling $W(T')$ by $(b-a)$ and translating it by $a$ moves these foci to $z=a$ and $z=b$, the eigenvalues of $A$. Similarly, scaling transforms the length of the minor axis to $|b-a|m=p= (\text{tr}(A^\star A) - |a|^2 - |b|^2)^{1/2}$, as needed.

\section{Proof of Theorem \ref{thm:E}} \label{sec:ERT2}

To prove Theorem \ref{thm:E}, we first show that $W(T')$ is the union of a family of circles. Then we use the envelope algorithm to show that the discriminant envelope of that family of circles equals $\partial \mathcal{E}$. Finally we use those results to show that $\overline{\mathcal{E}}$ is exactly the union of the circles, which proves Theorem \ref{thm:E}.     

For $t\in [0,1]$, let $\mathcal{C}_t$ denote the circle 
\begin{equation}\label{eqn:circles2}
	(x - (1 - t^2))^2+y^2 = m^2t^2(1-t^2).
\end{equation}
Then we prove the following elementary result:
\begin{lemma} \label{lem:Ct} The numerical range $W(T') = \bigcup_{t\in[0,1]} \mathcal{C}_t.$
\end{lemma}
\begin{proof} If $z\in \mathbb{C}^2$, then $\| z\| =1$ precisely when one can write $
z = \begin{bmatrix}te^{i\theta_1}\\ \sqrt{1-t^2}e^{i\theta_2}\end{bmatrix},$ for some $\theta_1, \theta_2 \in [0,2\pi], \text{ and } t\in [0, 1]$.
This formula for $z$ gives
$$
\langle T'z, z\rangle = (1 - t^2) + m e^{i\theta} (t \sqrt{1 - t^2}), \text{ where } \theta = (\theta_2 - \theta_1).
$$
This shows that each $(x,y)\in W(T')$ has the form
\begin{equation}\label{eqn:(x,y)}
x = (1 - t^2) + m \cos \theta  (t \sqrt{1 - t^2}) \
~\mbox{ and }~ 
\ y = m \sin \theta (t \sqrt{1 - t^2}),
\end{equation}
for some $t\in [0,1]$ and $\theta \in [0, 2\pi]$ and each such point is
in $W(T')$, as needed. \end{proof}

We now assume $m>0$ and apply the envelope algorithm to find the discriminant envelope of the family of circles in \eqref{eqn:circles2}. In fact, we prove:

\begin{lemma}\label{lemma:family}
	The discriminant envelope of the family of circles $\{ \mathcal{C}_t\}_{t\in[0,1]}$ is the union of the ellipse $\partial \mathcal{E}$ and the single point $(1,0).$ 
\end{lemma}

\begin{figure}[H]
\centering
	\includegraphics[width=8cm]{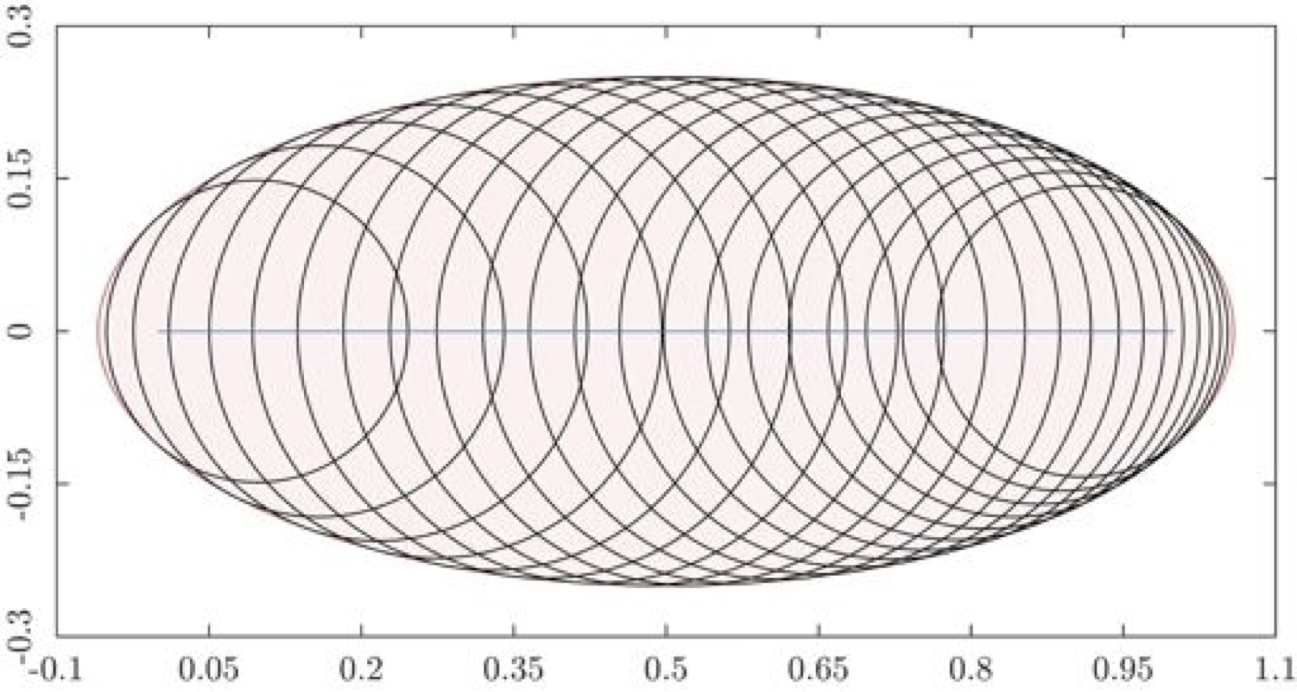}
\caption{\small A selection of circles $\mathcal{C}_t$ and the ellipse $\partial \mathcal{E}$ with $m=1.$}
\end{figure}

\begin{proof} First rewrite the circle equation from \eqref{eqn:circles2} as
\[F(x,y,t):=(x - (1 - t^2))^2+y^2-m^2t^2(1-t^2) = 0,\]
so  $\{ \mathcal{C}_t\}_{t\in[0,1]}$ is the family of curves satisfying $F(x,y,t)=0$.
 By definition, the discriminant envelope contains exactly the points $(x,y)$ that satisfy both $F(x,y,t)=0$ and $F_t(x,y,t)=0$ for some $t \in[0,1]$.  To find these points, compute $F_t(x,y,t)$ and observe that for $t=0$, $F_t(x,y,0)\equiv 0$ and $F(x,y,0)=0$ only at the point $(x,y)=(1,0)$. If $t\ne 0$, setting $F_t(x,y,t)=0$ and $F(x,y,t)=0$  gives the points $(x(t), y(t))$ with 
 \begin{equation}  \label{eqn:xy}
\begin{split}
x(t) &= (1 - t^2) + \tfrac{m^2}{2}(1 - 2 t^2),\\
y(t)  &= \pm \sqrt{ m^2 (t^2 - t^4) - \tfrac{m^4}{4}(1 - 2 t^2)^2}: = \pm \sqrt{g(t)}.
\end{split}
\end{equation}
We claim that on the subinterval(s) of $[0,1]$ where $y(t)$ makes sense, the points $(x(t), y(t))$ from  \eqref{eqn:xy} map out $\partial \mathcal{E}.$ To see this, observe that 
these $(x(t), y(t))$ satisfy
\[ \frac{(x(t) - \frac{1}{2})^2}{1 + m^2} + \frac{y(t)^2}{m^2} = \frac{(1 - 2t^2)^2}{4}(1 + m^2) + (t^2-t^4) - \frac{m^2}{4}(1 - 2t^2)^2 = \frac{1}{4}, \]
and so lie on $\partial \mathcal{E}.$

To see that $\partial \mathcal{E}$ equals the image of the functions from \eqref{eqn:xy}, observe 
that for $t$  sufficiently near $0$ or $1$, the value $g(t)$ in \eqref{eqn:xy} satisfies $g(t) <0$. As $g(\tfrac{1}{\sqrt{2}})>0$, there must be some $t_1, t_2$ satisfying $0 < t_1 < \frac{1}{\sqrt{2}} < t_2 <1$ with $g(t_1)=0=g(t_2).$ Then \eqref{eqn:xy} gives points  $(x(t_1), 0)$ and $(x(t_2), 0)$, where  $x(t_1) \ne x(t_2)$ because the function $x$ is strictly decreasing. As these points are also on $\partial \mathcal{E}$,  they must be the endpoints of the major axis: $(\tfrac{1}{2} \pm \tfrac{1}{2}\sqrt{1+m^2},0)$. As the functions in \eqref{eqn:xy} are continuous on $[t_1, t_2]$, it follows that $(x(t),  \sqrt{g(t)})$ maps out the top half of $\partial \mathcal{E}$ and $(x(t), -\sqrt{g(t)})$ maps out the bottom half of $\partial \mathcal{E},$ as needed.
\end{proof}

\begin{remark} As witnessed in the proof of Lemma \ref{lemma:family}, only some $\mathcal{C}_t$ contribute points to the envelope ellipse.  Analytically, circles corresponding to $t$-values sufficiently close to $0$ and $1$ do not contribute points because there are no solutions to the system of equations $F(x,y,t)=0$ and $F_t(x,y,t)=0$. Equivalently, the function $g(t)$ from \eqref{eqn:xy} is negative. If we let
 $t_1$ and $t_2$ denote the two unique solutions to $g(t)=0$ in $[0,1]$, then $(0,t_1)$ and $(t_2,1)$ are exactly the open intervals that do not contribute points to the envelope.
\begin{figure}[H]
\centering
	\includegraphics[width=8cm]{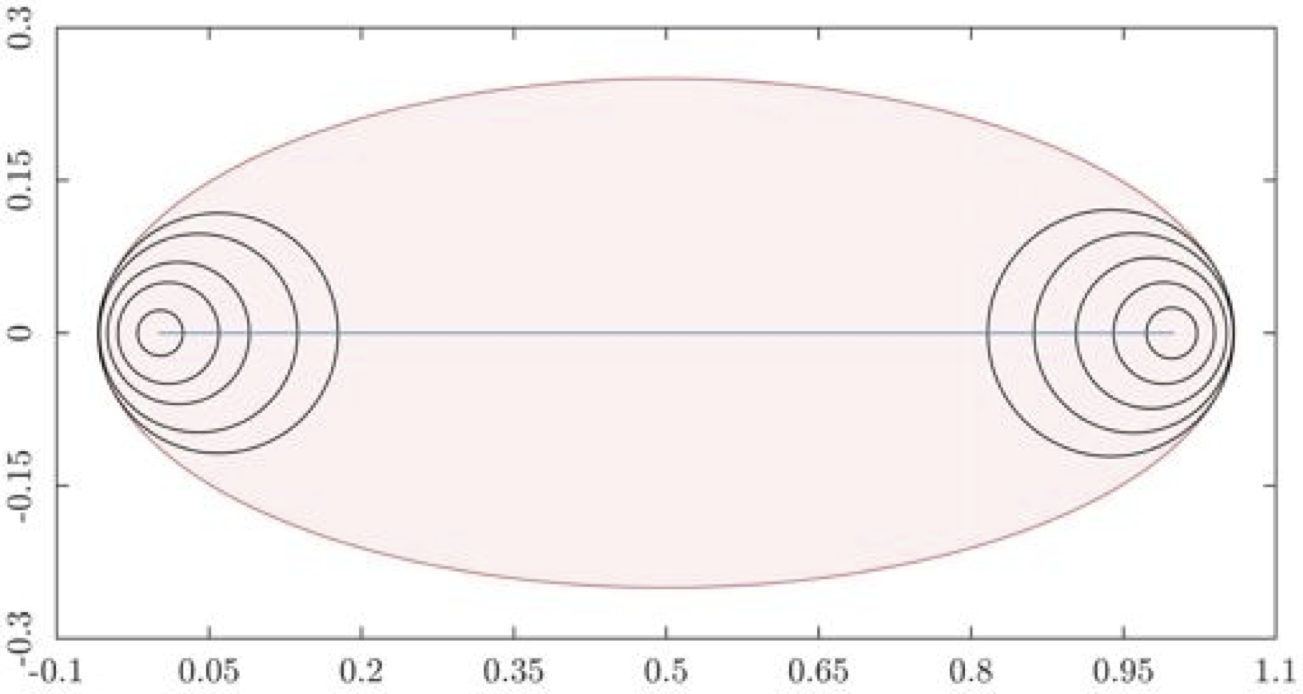}
\caption{\small Nested $\mathcal{C}_t$ that do not touch the envelope ellipse $\partial \mathcal{E}$.}
\end{figure}
Geometrically, the condition $g(t)<0$ means that the circle $\mathcal{C}_t$ does not intersect other $\mathcal{C}_{\tilde{t}}$ for $\tilde{t}$ close to $t$.  Indeed, the intervals $(0,t_1)$ and $(t_2,1)$ are exactly the open intervals where nearby circles are nested. 
\end{remark}

If we knew that the envelope curve $\partial \mathcal{E}$ equaled the boundary of the set $\bigcup_{t\in[0,1]} \mathcal{C}_t$, then a simple argument would imply $\bigcup_{t\in[0,1]} \mathcal{C}_t = \overline{\mathcal{E}}$ and yield Theorem \ref{thm:E}.  However in general, the envelope of a family of curves need not equal its boundary. Thus, we need some additional work. In particular, we first prove one direction of Theorem \ref{thm:E} via the following lemma:

\begin{lemma}\label{lemma:forward}
	Each $\mathcal{C}_t$ is contained in $\overline{\mathcal{E}}.$
\end{lemma}

\begin{proof} We show that every point $(x,y)$ from a $\mathcal{C}_t$ satisfies \eqref{eqn:ED}.  Fix $\theta \in [0, 2\pi]$ and let $L_{\theta}$ denote the left side of \eqref{eqn:ED}.
Then using the expressions for $x$ and $y$ from \eqref{eqn:(x,y)} yields
\begin{equation} \label{En}
 L_\theta(t) = \frac{\frac{1}{4}+(1-2t^2)m\cos(\theta)t\sqrt{1-t^2}+m^2(t^2-t^4)-\cos^2(\theta)(t^2-t^4)}{m^2+1}.
\end{equation}
We will show that $L_\theta(t)\leq \frac{1}{4}$ for $t\in [0,1]$. To this end, taking the derivative of $L_\theta$ yields
\begin{multline} \label{dEn/dt}
L'_\theta(t) = \tfrac{1}{\sqrt{1-t^2}(m^2+1)}\big[ -2t\sqrt{1-t^2}(1-2t^2)\cos^2(\theta) \\+ m(8t^4-8t^2+1)\cos(\theta)+2tm^2\sqrt{1-t^2}(1-2t^2) \big].
\end{multline}
To find the extrema, we set $L'_\theta(t)=0$ and solve to get two distinct values of $\cos(\theta)$
\begin{equation*}
\cos(\theta) = 
\begin{dcases}
\frac{2mt\sqrt{1-t^2}}{(2t^2-1)};\\
\frac{m(1-2t^2)}{2t\sqrt{1-t^2}}.\\
\end{dcases}
\end{equation*}
These formulas require that $t\not\in \{\tfrac{1}{\sqrt{2}}, 0, 1\}$, so we consider those values, which include the interval endpoints, later.  First, for $ \cos(\theta) = \frac{2mt\sqrt{1-t^2}}{(2t^2-1)}$, we have
\[
L_\theta(t) = \frac{4m^2t^4 - 4m^2t^2 +4t^4-4t^2+1}{4(2t^2-1)^2(m^2+1)}
=\frac{-\cos^2(\theta)+1}{4(m^2+1)} \leq \frac{1}{4}.
\]
Similarly for $ \cos(\theta) = \frac{m(1-2t^2)}{2t\sqrt{1-t^2}}$, we have 
$L_
\theta(t) = \frac{1}{4}.$ Finally, for $t=0, 1, \frac{1}{\sqrt{2}}$, \eqref{En} gives
$$
L_{\theta}(0) = L_{\theta}(1) = \tfrac{1}{4(m^2+1)} \leq \tfrac{1}{4} \ \ \text{ and } \ \
L_{\theta}\left(\tfrac{1}{\sqrt{2}}\right) = \tfrac{m^2-\cos^2(\theta)+1}{4(m^2+1)} \leq \tfrac{1}{4}.
$$
Since $L_\theta(t)\leq \frac{1}{4}$ for all $\theta \in [0, 2\pi]$ and $t\in [0,1]$, every $\mathcal{C}_t \subseteq \overline{\mathcal{E}}$, as needed.\end{proof}

We now complete the proof of Theorem \ref{thm:E} via the following lemma:

\begin{lemma}\label{lemma:backward}
The elliptical disk $\overline{\mathcal{E}}$ is contained in $\bigcup_{t\in[0,1]} \mathcal{C}_t.$
\end{lemma}

\begin{proof} In what follows, let $\mathcal{D}_t$ denote the open disk bounded by the circle $\mathcal{C}_t$. First by Lemma \ref{lemma:family}, $\partial \mathcal{E} \subseteq \bigcup_{t\in[0,1]} \mathcal{C}_t$.    Now let $(x_0, y_0) \in \mathcal{E}$ and for $t\in [0,1]$, define
\[ G(t): = F(x_0, y_0,t ) = (x_0 - (1 - t^2))^2+y_0^2-m^2t^2(1-t^2) .\]
 By the formulas in \eqref{eqn:xy}, there is a circle $\mathcal{C}_{t_0}$ that touches both points of $\partial \mathcal{E}$ whose $x$-coordinate is $x_0$. Then $(x_0, y_0) \in \mathcal{D}_{t_0}$ and so $G(t_0)<0$. If $(x_0,y_0) = (0,0)$, then $(x_0,y_0) \in \mathcal{C}_1$. Otherwise, $G(1) >0$. Then by continuity, there is some $\tilde{t} \in [t_0, 1]$ so that $G(\tilde{t}) =0$. This implies $(x_0, y_0) \in \mathcal{C}_{\tilde{t}}$ and so $\overline{\mathcal{E}}\subseteq \bigcup_{t\in[0,1]} \mathcal{C}_t$.
\end{proof}

\section{Families of Intersecting Circles} \label{sec:circles}
Now we consider more general families of intersecting circles. Specifically, given functions $x_c(t), y_c(t), r(t)$ for $t\in [s_1, s_2]$, let
\begin{equation} \label{eqn:circles} F(x,y,t) = (x-x_c(t))^2 +(y-y_c(t))^2 - r(t)^2.\end{equation}
For a fixed $t$, let $\mathcal{C}_t$  denote the circle defined by $F(x,y,t)=0$ and let $\mathcal{D}_t$  denote the open disk with boundary $\mathcal{C}_t$.
Finally, let $\mathcal{F}$ be the family of circles $\mathcal{C}_t$  for  $t \in [s_1, s_2]$. 

As before,  let $E_2$ denote the limiting-position envelope of $\mathcal{F}$ and let $E_3$ denote the discriminant envelope of $\mathcal{F}$. 
Then we have the following result relating the boundary of $\cup_t \mathcal{D}_t$ to these envelopes:

\begin{theorem} \label{thm:circle}Let $\mathcal{F}$ denote a family of circles parameterized by  $t \in [s_1, s_2]$ as in \eqref{eqn:circles}. Assume $x_c, y_c, r \in C^2([s_1, s_2])$ and 
for $t\in (s_1, s_2)$, $r(t) >0$ and 
\[ r'(t)^2 < x'_c(t)^2 + y_c'(t)^2.\]
Let $\Omega = \bigcup_{t\in[s_1,s_2]} \mathcal{D}_t.$ Then $\partial \Omega \subseteq E_2 \cup \mathcal{C}_{s_1} \cup \mathcal{C}_{s_2} \subseteq E_3 \cup \mathcal{C}_{s_1} \cup \mathcal{C}_{s_2}.$ \end{theorem}

To study the limiting-position envelope of a family of circles, we need the following information about intersections of circles. 

\begin{remark} \label{rem:intersection} Let $C_M$ be a circle with center $(x_M, y_M)$ and radius $r_M$ and $C_N$ a different circle with center $(x_N, y_N)$ and radius $r_N$. Then $C_M$ and $C_N$ intersect precisely when
\[ (r_M-r_N)^2 \le (x_M-x_N)^2+(y_M-y_N)^2 \le (r_M + r_N)^2.\]
If equality happens in this expression, i.e. if 
\[  (x_M-x_N)^2+(y_M-y_N)^2 = (r_M \pm r_N)^2,\]
 then the two circles are tangent. If $(x_M, y_M)$ is inside $C_N$, then the two circles are internally tangent. If $(x_M, y_M)$ is not inside $C_N$, then the $+$ indicates internally tangent circles and the $-$ indicates externally tangent circles. 

If the inequality is strict, there are two intersection points.
To find them, define
\[ 
\begin{aligned}
d &= \sqrt{ (x_M-x_N)^2 +(y_M-y_N)^2}, \\
K &= \tfrac{1}{4} \sqrt{ \left( (r_M +r_N)^2-d^2\right) \left(d^2- (r_M -r_N)^2\right)}.
\end{aligned}
\]
Then the points of intersection $(x_1, y_1)$ and $(x_2, y_2)$ of $C_M$ and $C_N$ are given by
\[ 
\begin{aligned}
x_j &= \tfrac{1}{2}(x_M +x_N) + \tfrac{1}{2d^2}(x_N-x_M)(r_M^2-r_N^2) +2(y_N-y_M)(-1)^{j+1} \tfrac{K}{d^2},\\
y_j &= \tfrac{1}{2}(y_M+y_N) + \tfrac{1}{2d^2}(y_N-y_M)(r_M^2-r_N^2) +2(x_N-x_M)(-1)^j \tfrac{K}{d^2},
\end{aligned}
\]
for $j=1,2.$
\end{remark}

We turn to the proof of Theorem \ref{thm:circle}:

\begin{proof}[Proof of Theorem \ref{thm:circle}.]  As $F$ and $F_t$ are continuous, \cite[Proposition 1]{bruce} implies that $E_2\subseteq E_3$ and so, we need only show that the $\partial \Omega \subseteq E_2 \cup \mathcal{C}_{s_1} \cup \mathcal{C}_{s_2}.$ \\

\noindent \textbf{Step 1: Show $\partial \Omega \subseteq \cup_t \mathcal{C}_t$}. We first observe that if $p \in \partial \Omega,$ then $p$ is in some $\mathcal{C}_t$. To see this, observe that if $p=(\tilde{x}, \tilde{y})\in \partial \Omega,$ then there is a sequence $(p_n) \subseteq \Omega$ converging to $p$. Then each $p_n=(x_n,y_n) \in \mathcal{D}_{t_n}$ for some $t_n \in [s_1, s_2]$ and by passing to a subsequence, we can assume that $(t_n)$ converges to some $\tilde{t}\in [s_1, s_2]$. As
\[  (x_n-x_c(t_n))^2 +(y_n-y_c(t_n))^2 < r(t_n)^2\]
for each $n$, we can let $n\rightarrow \infty$ and use continuity to conclude that 
\[  (\tilde{x}-x_c(\tilde{t}))^2 +(\tilde{y}-y_c(\tilde{t}))^2 \le r(\tilde{t})^2.\]
Thus, $p \in \overline{\mathcal{D}}_{\tilde{t}} \cap \partial \Omega$, which implies  $p \in \mathcal{C}_{\tilde{t}}$. \\

In the remainder of the proof, we will show that for each $t\in (s_1, s_2)$, the circle $\mathcal{C}_t$ contributes at most two points to $\partial \Omega$ and those points are both in $E_2.$ Once we establish this, it is immediate that $\partial \Omega \subseteq  \cup_t (\partial \Omega\cap \mathcal{C}_t) \subseteq E_2 \cup \mathcal{C}_{s_1} \cup \mathcal{C}_{s_2}.$
As establishing the result about $\mathcal{C}_t$ is a local argument, we will now (and throughout the rest of the proof) fix $t \in (s_1, s_2)$, assume this fixed $t$-value is $0$, and prove that $\mathcal{C}_0$ has the desired properties. \\

\noindent \textbf{Step 2: Find circle intersection points and their limits}. 
Let $(\hat{x}, \hat{y}):=(x_c(0), y_c(0))$ and $\hat{r}:=r(0).$  Our smoothness assumptions paired with Taylor's Theorem imply that for $t$ sufficiently close to $0$, there are $A,B,C\in \mathbb{R}$, equal to $x'(0), y'(0), r'(0)$ respectively with
\[ 
x_c(t) =  \hat{x} + At + O(t^2), \ \ \ y_c(t) =  \hat{y} + Bt + O(t^2), \ \ \
r(t) =  \hat{r} + Ct + O(t^2).
\]
Set $D = \sqrt{A^2 +B^2}.$ The assumptions that $C^2 < D^2$ and $\hat{r}>0$ imply that for $t$ sufficiently small, 
\[ (\hat{r} - r(t))^2 < (x_c(t) - \hat{x})^2 + (y_c(t)- \hat{y})^2 < (\hat{r} + r(t))^2,\]
so the circles  $\mathcal{C}_t$ and $\mathcal{C}_0$ intersect in two distinct points. Then Remark \ref{rem:intersection} implies that the intersection points $P_1^t=(x_1^t, y_1^t)$ and $P_2^t=(x_2^t, y_2^t)$ satisfy the formulas
\[ 
\begin{aligned}
x^t_j &= \hat{x} + \tfrac{1}{2}At + O(t^2) - \tfrac{AC\hat{r} +O(t) }{D^2 + O(t)} +\hat{r}(-1)^{j+1} \tfrac{|t|}{t}(B +O(t)) \tfrac{\sqrt{D^2-C^2 +O(t)}}{D^2 + O(t)},\\
y^t_j &= \hat{y} + \tfrac{1}{2}Bt + O(t^2) - \tfrac{BC\hat{r} +O(t) }{D^2 + O(t)} +\hat{r}(-1)^j \tfrac{|t|}{t}(A +O(t)) \tfrac{\sqrt{  D^2-C^2 +O(t)}}{D^2 + O(t)},
\end{aligned}
\]
for $j=1,2$. Define the points $P_j=(x_j^p, y_j^p)$ by 
\[
\begin{aligned}
x_j^p &= \hat{x} +\tfrac{\hat{r}}{D^2}\left(- AC  +(-1)^{j+1} B\sqrt{D^2-C^2}\right),\\
 y_j^p &= \hat{y} +\tfrac{\hat{r}}{D^2}\left( -BC +(-1)^jA\sqrt{D^2-C^2}\right). 
\end{aligned}
\]
As $\hat{r}\ne0$ and $A, B$ are not simultaneously $0$, it is clear that $P_1 \ne P_2$. Furthermore, if $t >0$ and $t \searrow 0$, one can see that $(x_1^t, y_1^t) \rightarrow P_1$ and $(x_2^t, y_2^t) \rightarrow P_2$. Similarly, if $t<0$ and $t \nearrow 0$, one can see that $(x_1^t, y_1^t) \rightarrow P_2$ and $(x_2^t, y_2^t) \rightarrow P_1$. This implies that $P_1, P_2$ are in  $E_2$ and as they are limits of points on $\mathcal{C}_0$, they are also on $\mathcal{C}_0$.
\\

\noindent \textbf{Step 3: Identify arcs from $\mathcal{C}_0$ contained in other $\mathcal{D}_t$}. 
By Step $2$, for each $t$ near $0$, the open disk $\mathcal{D}_t$ contains an open arc $\mathcal{I}_t \subseteq \mathcal{C}_0$ with endpoints given by the intersection points $P_1^t$ and $P_2^t$. In general, there are two choices for $\mathcal{I}_t$. We will uniquely identify this arc by specifying points that it does and does not contain.
 
First consider the diameter of $\mathcal{C}_0$ that lies along the line $(\hat{x}, \hat{y}) + s(A,B)$ parametrized by $s$. Then the endpoints of this diameter are given by
\[ 
Q_1 = (\hat{x}, \hat{y}) + \tfrac{\hat{r}}{D}(A, B) \ \ \text{ and } \ \ 
Q_2 = (\hat{x}, \hat{y}) - \tfrac{\hat{r}}{D}(A, B). 
\]
We claim that for $t>0$ sufficiently small, $Q_1 \in \mathcal{D}_t$ and $Q_2 \not \in \mathcal{D}_t$. First observe that
\[ 
\begin{aligned}
\left | (x_c(t), y_c(t)) - Q_1 \right|^2 &= \left( At -  \tfrac{\hat{r}A}{D} +O(t^2)\right)^2 + \left( Bt -  \tfrac{\hat{r}B}{D} +O(t^2)\right)^2 \\
&= \hat{r}^2 -2 t\hat{r}D +O(t^2), 
\end{aligned}
\]
where $\hat{r}$ is the radius of $\mathcal{C}_0.$
Then as $t >0$ and $-C\le |C|< D$, we have
\[ \left | (x_c(t), y_c(t)) - Q_1 \right|^2 - r(t)^2 = -2t\hat{r}\left( D +C\right) +O(t^2) <0,\]
for $t$ sufficiently small. Thus  $Q_1 \in \mathcal{D}_t.$ Similarly, 
\[ 
\left | (x_c(t), y_c(t)) - Q_2 \right| - r(t)^2 = 2t\hat{r} \left( D -C\right) +O(t^2) >0,
\]
which implies $Q_2 \not \in \mathcal{D}_t$ for $t$ sufficiently small. Identical arguments show that  for small $t<0$, $Q_2 \in \mathcal{D}_t$ and $Q_1 \not \in \mathcal{D}_t.$

Thus, if $t$ is small and positive, then $\mathcal{D}_t$ contains an open arc $\mathcal{I}_t \subseteq \mathcal{C}_0$ with endpoints  $P_1^t$ and $P_2^t$ that contains $Q_1$ and not $Q_2$. If $t$ is small and negative, then $\mathcal{D}_t$ contains an open arc $\mathcal{I}_t \subseteq \mathcal{C}_0$ with endpoints  $P_1^t$ and $P_2^t$ that contains $Q_2$ and not $Q_1$. 

Finally, note that $\{P_1, P_2\} \ne \{Q_1, Q_2\}$. Indeed if one sets $P_1=Q_1$ and $P_2=Q_2,$ the resulting formulas imply that $A=B=C=0$, a contradiction. The same contradiction arises if one sets $P_1=Q_2$ and $P_2=Q_1.$ Thus without loss of generality, we can assume $ Q_1\not \in \{P_1, P_2\}$. 

\begin{figure}[H]
	\centering
	\includegraphics[width=8cm]{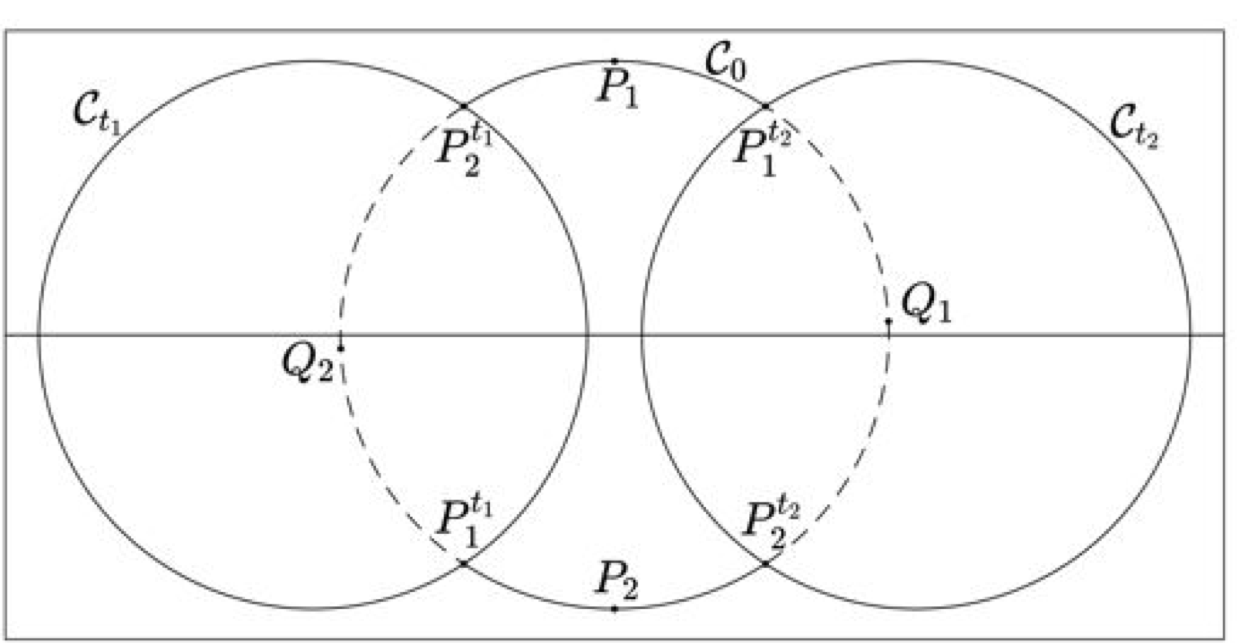}
\caption{\small Example of circles $\mathcal{C}_{t_1}, \mathcal{C}_0, \mathcal{C}_{t_2}$ with $t_1<0<t_2$ and the points $P_1, P_2, Q_1,Q_2.$ }
\end{figure}

\noindent \textbf{Step 4: Show $\partial \Omega \cap \mathcal{C}_0 \subseteq \{P_1,P_2\}$.} 
Let $\mathcal{I}_{+}, \mathcal{I}_-$ denote the two distinct open arcs of $\mathcal{C}_0$ with endpoints $P_1$ and $P_2$, chosen so $Q_1 \in \mathcal{I}_+$. For $p, q \in\mathcal{C}_0$, let $|p-q |_{0}$ denote the length of the shortest arc between them on $\mathcal{C}_0$. 
  We show that $\left( \mathcal{I}_+\cup \mathcal{I}_- \right) \cap \partial \Omega = \emptyset.$ 
  
Fix $p \in \mathcal{I}_+$. We claim $p \in \mathcal{D}_t$ for some $t>0$. Let $\epsilon = \min_j \{|p-P_j|_0, |Q_1-P_j|_0\} >0$. By Step $2$, we can choose $t>0$ sufficiently small  so that  $\min_j |P_j^t- P_j|_0 <\frac{\epsilon}{3}.$ Then $p \in \mathcal{I}_t$, the open arc in $\mathcal{C}_0$ from $P_1^t$ to $P_2^t$ containing $Q_1$. 
By Step $3$, $\mathcal{I}_t \subset \mathcal{D}_t$ and so $p \not \in \partial \Omega.$

Similarly, each $p \in \mathcal{I}_-$ is in some $\mathcal{D}_t$ with $t<0$. To see this, set $\epsilon = \min_j \{ |p- P_j|_0, |Q_1-P_j|\} >0$ and choose 
$t<0$ sufficiently small so that   $\min\{ |P_1^t-P_2|, |P_2^t-P_1|\} <\frac{\epsilon}{3}.$ Then $p \in \mathcal{I}_t$, the open arc in $\mathcal{C}_0$ from $P_1^t$ to $P_2^t$ that does not contain $Q_1$. By Step $3$, $\mathcal{I}_t \subset \mathcal{D}_t$ and so $p \not \in \partial \Omega.$ This gives $\partial \Omega \cap \mathcal{C}_0 \subseteq \{P_1,P_2\}$, which completes the proof.
\end{proof}

\begin{remark} \label{rem:formulas} This argument also gives a parameterization of the points of $E_2$ that can contribute to $\partial \Omega.$ Specifically, the points in $E_2$ (except for those coming from $\mathcal{C}_{s_1}$ and $\mathcal{C}_{s_2})$ that can contribute to $\partial \Omega$ are exactly those of the form $(x_1(t), y_1(t))$, $(x_2(t), y_2(t))$ where
\[
\begin{aligned}
x_j(t) &= x_c(t) +\tfrac{r(t)}{x_c'(t)^2+y_c'(t)^2}\left(- x_c'(t)r'(t)  +(-1)^{j+1} y_c'(t)\sqrt{x_c'(t)^2+y_c'(t)^2-r_c'(t)^2}\right),\\
 y_j(t) &= y_c(t) +\tfrac{r(t)}{x_c'(t)^2+y_c'(t)^2}\left( -y_c'(t)r'(t) +(-1)^j x_c'(t)\sqrt{x_c'(t)^2+y_c'(t)^2-r_c'(t)^2}\right),
\end{aligned}
\]
 for $j=1,2$ and $t\in(s_1, s_2).$
\end{remark}

In the next two sections, we apply Theorem \ref{thm:circle} to families of pseudohyperbolic disks.

\section{Pseudohyperbolic Disks on a Line} \label{sec:line}

In this section, we fix $r\in (0,1)$ and consider the family of pseudohyperbolic disks $\{ D_\rho(t,r): t \in [-1,1]\}$, whose centers lie along a line segment. For each $t\in [-1,1]$, let $\overline{D_\rho(t,r)}$ denote the closure of $D_\rho(t,r)$ and $\mathcal{S}_t$  its boundary circle. Then by \eqref{eqn:cR}, $\mathcal{S}_t$ is defined by the equation
\begin{equation} \label{[eqn]:S_t}
\left(x - \frac{1-r^2}{1-r^2t^2}t\right)^2 + y^2 = \left(\frac{r(1-t^2)}{1-r^2t^2}\right)^2.
\end{equation}
 From \eqref{eqn:disks}, recall the disks 
\[\mathcal{D}_1=\{z \in \mathbb{C} :\left|z+ \tfrac{1-r^2}{2r}i \right| < R_r \} \ \ \text{ and }  \ \ \mathcal{D}_2=  \{z \in \mathbb{C} :\left |z-\tfrac{1-r^2}{2r}i \right| < R_r \},\]
with radius $R_r:=\frac{1+r^2}{2r}$ and from \eqref{eqn:arcs}, recall the arcs $\mathcal{A}_1= \partial \mathcal{D}_1 \cap \overline{\mathbb{H}^+}$ and $\mathcal{A}_2= \partial \mathcal{D}_2 \cap \overline{\mathbb{H}^-}.$
Then we show
\begin{theorem} \label{thm:bdy} The union $\bigcup_{t\in[-1,1]} D_\rho(t,r)$ has boundary $\mathcal{A}_1 \cup \mathcal{A}_2$.  \end{theorem}

:
\begin{figure}[H]
	\centering
	\includegraphics[width=8cm]{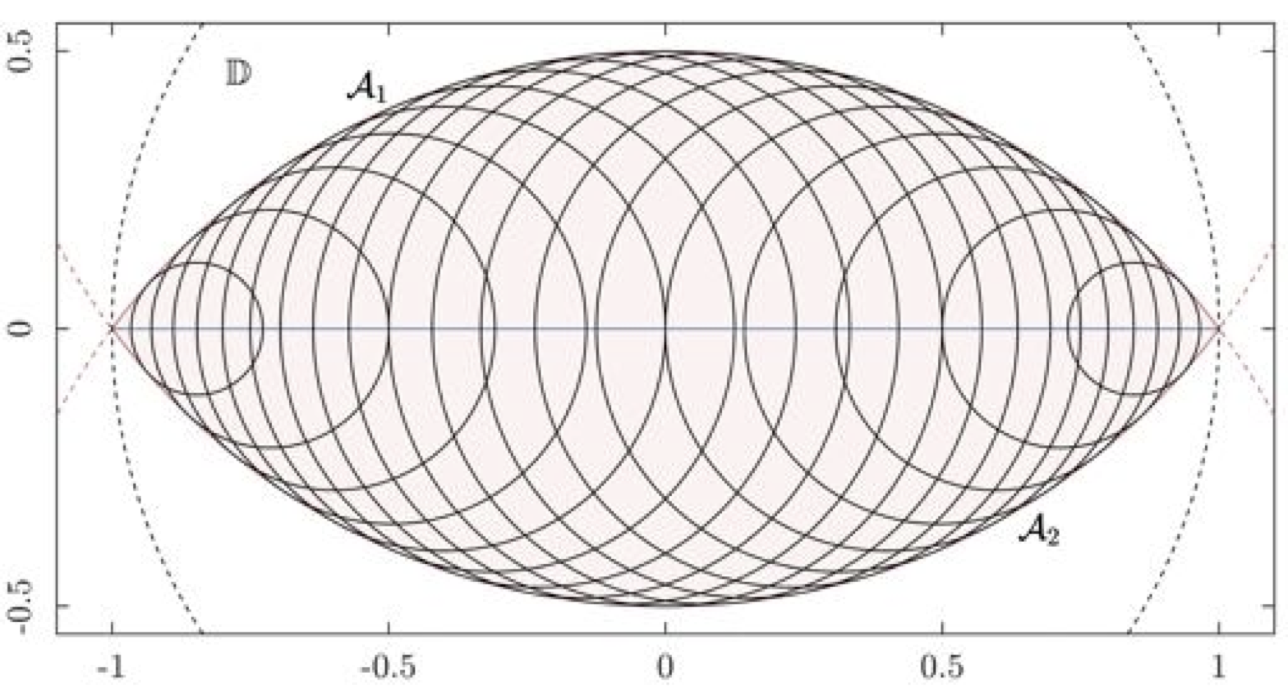}
\caption{\small The family $\{ D_\rho(t,r)\}_ {t \in [-1,1]}$ with boundary $\mathcal{A}_1 \cup \mathcal{A}_2$. }
\end{figure}

To prove Theorem \ref{thm:bdy}, we will use Theorem \ref{thm:circle}. This requires the following lemma: 

\begin{lemma} \label{lem:envelope} The discriminant envelope $E_3$ of the family of circles $\{ \mathcal{S}_t\}_{t\in[-1,1]}$ is $\mathcal{A}_1 \cup\mathcal{A}_2$.
\end{lemma}

\begin{proof} Fix $r \in (0,1)$ and rewrite the circle equation from \eqref{[eqn]:S_t} as
$$
F(x,y,t)=\Bigg(x - \frac{1-r^2}{1-r^2t^2}t\Bigg)^2 + y^2 - \Bigg(\frac{r(1-t^2)}{1-r^2t^2}\Bigg)^2 = 0.
$$
Now apply the envelope algorithm. Differentiating with respect to $t$, setting $F_t(x,y,t)=0$, and requiring that $(x, y)$ also satisfy $F(x,y,t)=0$ yields the points
\[
x(t) = \frac{(r^2+1)t}{1+r^2t^2} \ \ \text{ and } \ \ y(t) = \pm \frac{r(1-t^2)}{1+r^2t^2},
\]
for $t\in[-1,1].$ This set of points is precisely the discriminant envelope $E_3.$  Define 
\begin{equation}
 \label{eqn:env1} 
E_3^+ : =  E_3 \cap \overline{\mathbb{H}^+} = \left\{ (x(t), y(t)) =\left( \frac{(r^2+1)t}{1+r^2t^2}, \frac{r(1-t^2)}{1+r^2t^2}\right) :  t\in [-1,1]\right\}. \end{equation}
We show that $E_3^+$  equals the curve $\mathcal{A}_1$. 
To this end, assume $(x(t),y(t)) \in E_3^+$. Then $y(t)\geq 0$ and 
$$ 
x(t)^2 + \Big(y(t) + \tfrac{1-r^2}{2r}\Big)^2 = \tfrac{(1+r^2)^2}{4r^2},
$$
so $(x(t),y(t)) \in \overline{\mathbb{H}^+}\cap \partial \mathcal{D}_1$ and $E_3^+ \subseteq \mathcal{A}_1$. Further, 
at $t=\pm1$, the formulas in  \eqref{eqn:env1} give the two endpoints of $\mathcal{A}_1$: $(\pm 1,0).$ As the formulas in
 \eqref{eqn:env1} are also continuous in $t$, they must map out the entire curve $\mathcal{A}_1$.

To complete the proof, define $E_3^-=E_3 \cap \overline{\mathbb{H}^-}$ and use identical arguments to show that $E_3^- = \mathcal{A}_2.$ Thus the discriminant envelope $E_3= \mathcal{A}_1 \cup \mathcal{A}_2$, as needed.
\end{proof}

We turn to the proof of Theorem \ref{thm:bdy}.

\begin{proof}[Proof of Theorem \ref{thm:bdy}.] Let $\Omega =  \cup_{t\in[-1,1]} D_\rho(t,r)$. We first show that the family of circles $\{\mathcal{S}_t\}_{t\in[-1,1]}$ satisfies the conditions of Theorem \ref{thm:circle}. In particular, note that 
\[ x_c(t) = \frac{1-r^2}{1-r^2t^2}t, \ \ \ y_c(t) = 0, \ \ \  r(t) = \frac{r(1-t^2)}{1-r^2t^2}. \]
These functions are in $C^2([-1,1])$ and for $t\in (-1,1)$, the radius $r(t) >0.$ Moreover, as
\[ x_c'(t) = \frac{(1-r^2)(r^2t^2+1)}{(1-r^2t^2)^2} \ \ \text{ and } \ \ r'(t) = \frac{2rt(r^2-1)}{(1-r^2t^2)^2},\]
one can easily compute
\[ x_c'(t)^2 +y'_c(t)^2 - r'(t)^2= \frac{(1-r^2)^2}{(1-r^2t^2)^2}>0.\]
Thus $\{\mathcal{S}_t\}_{t\in[-1,1]}$ satisfies the conditions of Theorem \ref{thm:circle}, so Lemma \ref{lem:envelope} implies that 
\[\partial \Omega \subseteq \mathcal{A}_1 \cup\mathcal{A}_2\cup S_{-1} \cup S_1 = \mathcal{A}_1 \cup \mathcal{A}_2,\]
where we used the fact that $\mathcal{S}_{-1} = (-1,0)$ and $\mathcal{S}_1 = (1,0)$ are in $\mathcal{A}_1$ and $\mathcal{A}_2.$

To show the other containment, recall that $\mathcal{A}_1 \subset \overline{\mathbb{H}^+}$ and $\mathcal{A}_2\subset \overline{\mathbb{H}^-}$ are arcs of circles that intersect at $(\pm 1,0)$. 
Without loss of generality, assume some $p \in \mathcal{A}_1$ is not in $\partial \Omega.$ By Lemma \ref{lem:envelope}, $p \in \mathcal{S}_t$ for some $t$.
This implies that there is some open disk $B_{\epsilon}(p)$, with center $p$ and radius $\epsilon >0,$ that is contained in $\Omega.$ Let $(v_1, v_2)$ denote a vector normal to $\mathcal{A}_1$ at $p$ pointing out of $\mathcal{D}_1$. 
Let $\tilde{s} = \sup\{ s\in \mathbb{R}: p+ s(v_1, v_2) \in \Omega\}.$ Then $\tilde{p} =  p+ \tilde{s}(v_1, v_2) \in \partial \Omega$. However as $\tilde{s} \ge \epsilon$, the simple geometry of $\mathcal{A}_1 \cup \mathcal{A}_2$ implies that $\tilde{p} \not \in \mathcal{A}_1 \cup \mathcal{A}_2$, a contradiction.
\end{proof}
As mentioned earlier, this result has applications to various aspects of function theory on the disk. These are detailed in the following remark.

\begin{remark}\label{rem:app1}Theorem \ref{thm:bdy} says that, in particular, the points in the union of pseudohyperbolic disks $\cup_{t\in[-1,1]} D_\rho(t,r)$ lie in a Stolz region, where a Stolz angle at a point $\lambda \in \mathbb{T}$ is defined as the set of points \[S(\lambda) = \{z \in \mathbb{D}: |1 - \overline{\xi}z| \le K (1 - |z|)\},\] where $K$ is a fixed constant greater than $1$. A sequence of points in such a region that converges to $\lambda$ is said to approach the boundary nontangentially. 

There are many important applications of this. For example, consider the following: Let $\mathcal{A}^p$ denote the Bergman space of all functions analytic on $\mathbb{D}$ such that 
\[\|f\| = \left(\int_\mathbb{D} |f(z)|^p dA(z)\right)^{1/p} < \infty,\] where the integral is with respect to area measure on $\mathbb{D}$. If the zeros of a Blaschke product $B$ all lie in a Stolz angle, then the derivative $B^\prime$ lies in $\mathcal{A}^p$ for $p < 3/2$, \cite{GPV07}. 

Interpolating sequences for the algebra $H^\infty$ of bounded analytic functions on $\mathbb{D}$ provide another example of the important role Stolz angles play. A sequence of points $(z_n)$ in $\mathbb{D}$ is an interpolating sequence for $H^\infty$ if for every bounded sequence $(w_n)$ of complex numbers, there is a function $f \in H^\infty$ such that $f(z_n) = w_n$ for all $n$. Carleson showed that a sequence is interpolating if and only if it is uniformly separated; that is,
\[\inf_k \prod_{j: j \ne k} \left| \frac{z_j - z_k}{1 - \overline{z_j} z_k}\right| \ge \delta>0.\]  A sequence $(z_n)$ in $\mathbb{D}$ is said to be an exponential sequence if there is a constant $c < 1$ such that for all $n > 1$
\[\frac{1 - |z_n|}{1 - |z_{n-1}|} < c.\] This is a strong condition on a sequence: Hayman and D. J. Newman \cite[p. 203]{Ho} both showed that an exponential sequence is interpolating. In \cite{N59} Newman notes that a uniformly separated sequence of points in a Stolz angle is a finite union of exponential sequences. This result is extended in \cite{GPV08}. In this setting, these results provide an alternative to determining whether a sequence is interpolating from the uniform separation constant: We say that a sequence $(z_n)$ is separated if 
\[\inf_n \left|\frac{z_n - z_{n+1}}{1 - \overline{z_n}z_{n+1}}\right| \ge \delta > 0;\] in other words, there is a minimal positive distance between successive points in the pseudohyperbolic mentric. To check that a sequence in the union of pseudohyperbolic disks  is interpolating, we need only check the separation condition -- a much easier condition to check.
\end{remark}

\section{Pseudohyperbolic Disks on a Horocycle} \label{sec:horocycle}

In this section, we fix $r\in (0,1), k>0$ and consider the family of pseudohyperbolic disks $\{D_\rho(\alpha,r): \alpha \in H(1, k)\},$
whose centers lie along the horocycle $H(1,k)$ from \eqref{eqn:horo}.  Recall the open disks
\[ \mathscr{D}_1= \left\{z \in \mathbb{C} :\left |z- c_1\right| < R_1 \right\} \ \ \text{ and }  \ \ \mathscr{D}_2=  \{z \in \mathbb{C} :\left|z-c_2\right| < R_2 \} ,
\] from \eqref{eqn:disks2}, where the $c_j, R_j$ are from \eqref{eqn:c1c2}.

Observe that as $(c_1-c_2)^2 - (R_1-R_2)^2=0$, Remark \ref{rem:intersection} implies that $\partial \mathscr{D}_1$ and $\partial \mathscr{D}_2$ are tangent to one another. The point of tangency is $(1,0)$ and as $R_2<R_1$, one can show that $c_2 \in \mathscr{D}_1$ and conclude that $\mathscr{D}_2 \subseteq \mathscr{D}_1.$
We will prove:

\begin{theorem} \label{thm:bdy2} The union $\bigcup_{\alpha\in H(1, k)} D_\rho(\alpha,r)$ has boundary $\partial \mathscr{D}_1 \cup \partial \mathscr{D}_2$.   \end{theorem}

\begin{figure}[H]
	\centering
	\includegraphics[width=4cm]{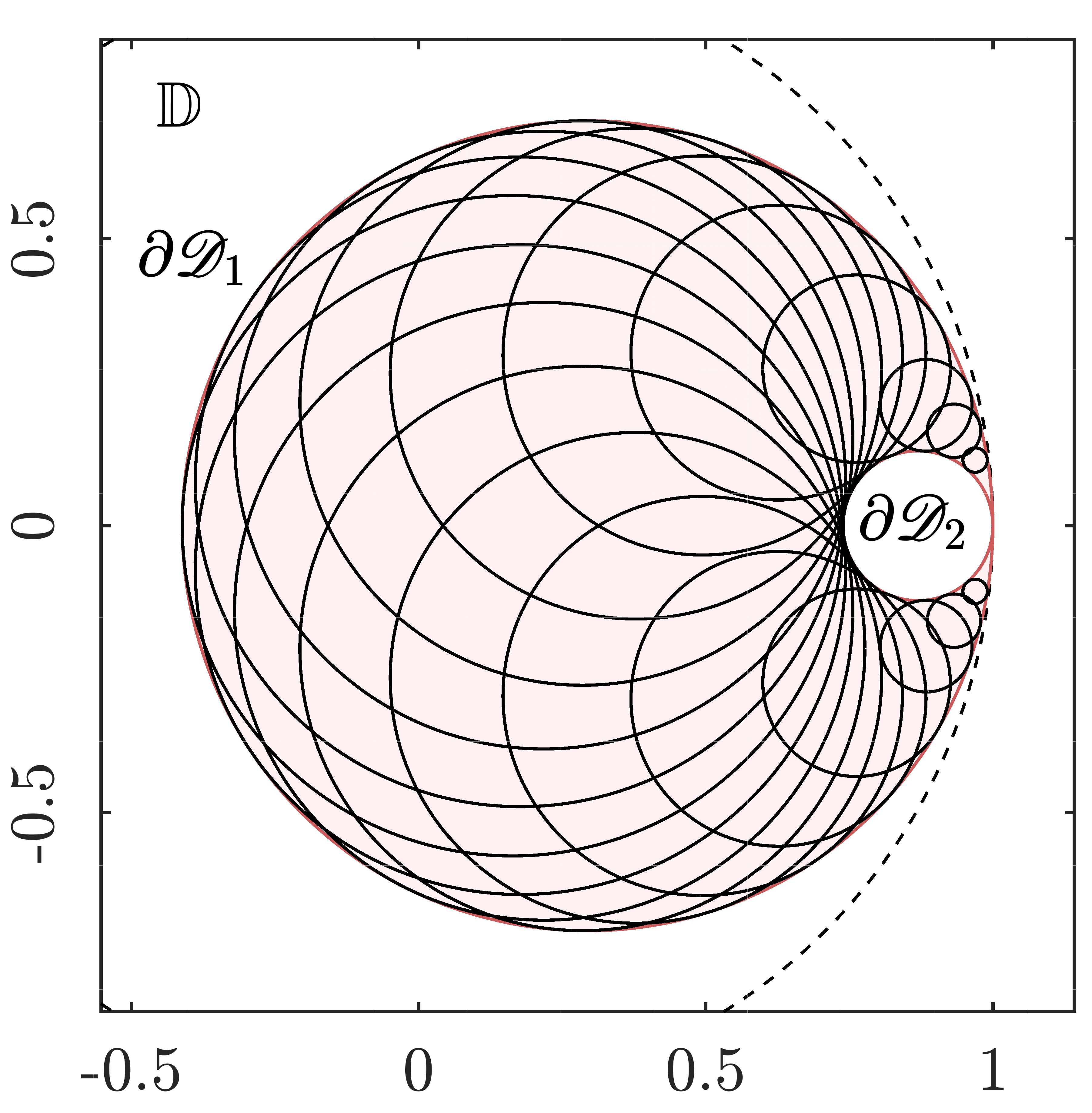}
\caption{\small The family $\{ D_\rho(\alpha,r)\}_ {\alpha \in H(1,k)}$ with boundary $\partial \mathscr{D}_1 \cup \partial \mathscr{D}_2$. }
\end{figure}

Define $a = \frac{1}{2}(R_1 + R_2)$, $b = \sqrt{R_1 R_2}$, and $c = \frac{1}{2}(c_1+c_2)$. Let $\mathcal{E}$ denote the ellipse 
	\begin{equation}\label{eqn:E} \tfrac{(x-c)^2}{a^2} + \tfrac{y^2}{b^2} =1. \end{equation}
We first use these constants to rewrite the family of pseudohyperbolic circles as a family of Euclidean circles $\{\mathcal{S}_t\}_{t\in[0, 2\pi]}$ whose centers lie along $\mathcal{E}.$ 


\begin{theorem} \label{thm:euclidean} The set of Euclidean centers $\left\{c_\rho(\alpha): \alpha \in H(1,k) \right\}$ equals the ellipse $\mathcal{E}$ in \eqref{eqn:E}. The family of pseudohyperbolic circles $\{\partial D_\rho(\alpha,r)\}_{\alpha \in H(1, k)}$ equals the family of Euclidean circles $\{\mathcal{S}_t\}_{t\in[0, 2\pi]}$, where each $\mathcal{S}_t$ is defined by 
  \begin{equation} \label{eqn:St}
\left(x - (c+a\cos t) \right)^2 + (y-b\sin t )^2 = 
\left(\tfrac{R_1-R_2}{2} \right)^2 (1-\cos t )^2.
\end{equation}
\end{theorem}

\begin{figure}[H]
	\centering
	\includegraphics[width=4cm]{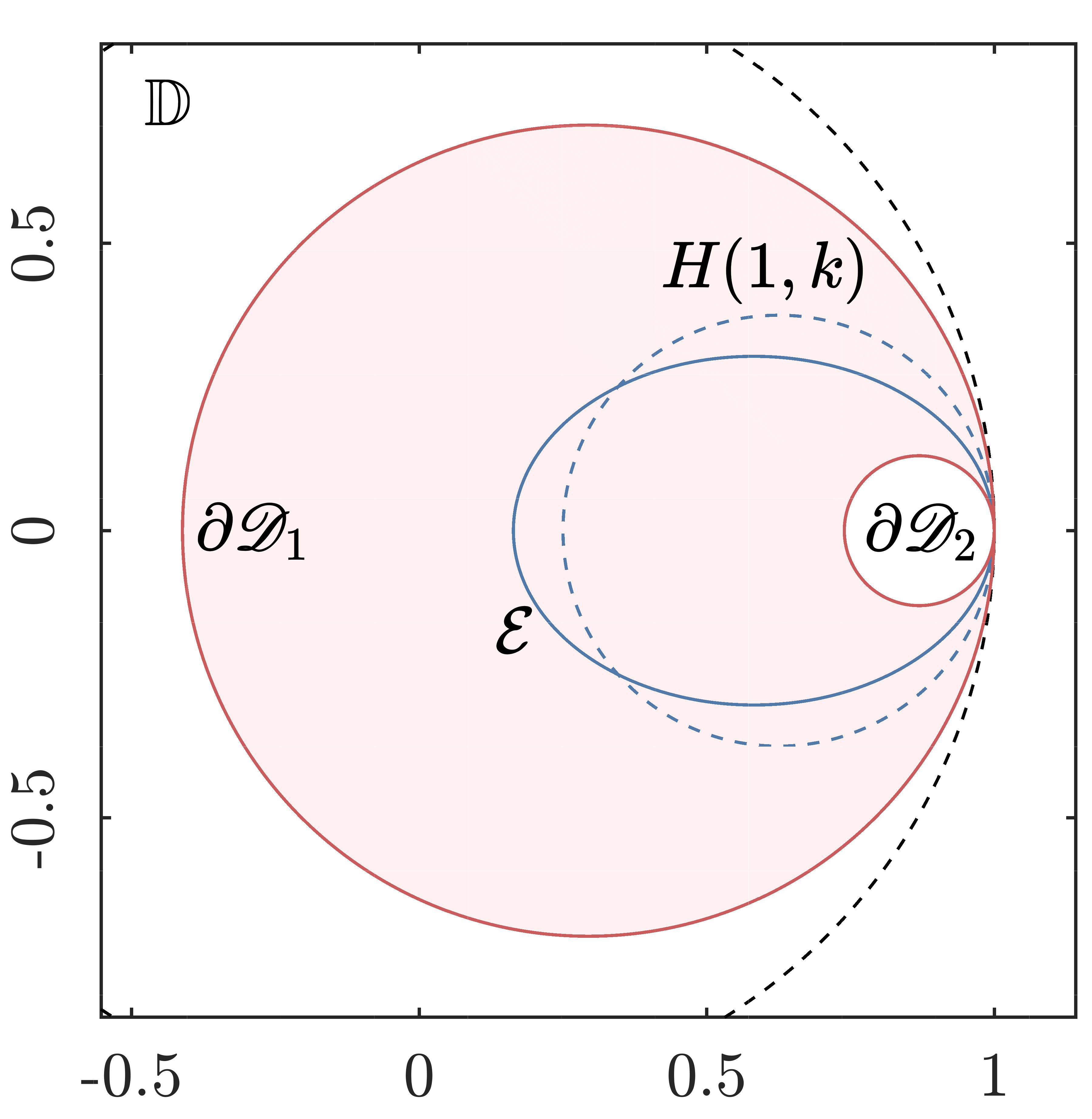}
\caption{\small The ellipse $\mathcal{E}$ of Euclidean centers with $H(1,k)$ and  $\partial \mathscr{D}_1 \cup \partial \mathscr{D}_2$. }
\end{figure}

 \begin{proof}
 	First we show that $\left\{c_\rho(\alpha): \alpha \in H(1,k) \right\} \subseteq \mathcal{E}$, where $\mathcal{E}$ is the ellipse defined in \eqref{eqn:E}.
	For each $\alpha  \in H(1,k)$, there is a $\gamma \in [0, 2\pi)$ so that $\alpha = \frac{1}{k+1} + \frac{k}{k+1} e^{i\gamma}.$ Using this in  
 	\eqref{eqn:cR} gives 
 	\begin{equation*}
	c_\rho(\alpha)  = x_{c_\rho}(\gamma)+i y_{c_\rho}(\gamma)= \frac{(1-r^2)(k+1)(1 + ke^{i \gamma})}{(1+k^2)(1-r^2)+2k(1-r^2\cos \gamma) }.\end{equation*}
A computation shows that each $ (x_{c_\rho}(\gamma), y_{c_\rho}(\gamma))$ satisfies \eqref{eqn:E} 
 	and therefore each $c_\rho(\alpha) \in \mathcal{E}$. To see the other containment, note that $\gamma=0, 2\pi$ both yield the points $(1,0)$ and $\gamma =\pi$ yields the other point on $\mathcal{E} \cap \mathbb{R}.$ As 	
	$y_{c_\rho}(\gamma)> 0$ for all $\gamma \in (0,\pi)$ and $y_{c_\rho}(\gamma)< 0$ for all $\gamma \in (\pi,2\pi)$, continuity implies that $(x_{c_\rho}(\gamma), y_{c_\rho}(\gamma))$ maps out $\mathcal{E}$.

	Now fix a pseudohyperbolic circle $\partial D_\rho(\alpha,r)$ with $\alpha \in H(1,k)$. Then its Euclidean center is on $\mathcal{E}$ and so can be written as 
	\[ c_{\rho} = (c +a\cos t, b \sin t)\]
	for some $t \in [0, 2 \pi]$, and each such $t$ yields the center of some $\partial D_\rho(\alpha,r)$. We claim that the Euclidean radius $R_{\rho}$ of this $\partial D_\rho(\alpha,r)$ is given by 
	\[\tilde{R}:=\tfrac{R_1-R_2}{2} (1-\cos t).\]
	First, note that \eqref{eqn:cR} implies that the moduli $|\alpha|$ is determined by $|c_{\rho}|$, i.e. $|\alpha|$ is the unique nonnegative number satisfying 
	\[ |c_\rho| = \frac{(1-r^2) |\alpha|}{1-r^2|\alpha|^2}.\]
	 Similarly,  \eqref{eqn:cR}  gives a one-to-one correspondence between $R$ satisfying $0<R\le r$ and $|\beta|^2$ with $0\le |\beta| <1.$
	By way of contradiction, assume that $\tilde{R} \ne R_{\rho}$. As $0<\tilde{R}\le r$, \eqref{eqn:cR} gives a unique associated 
	$|\beta_{\tilde{R}}|^2 = \frac{r-\tilde{R}}{r(1-r\tilde{R})}.$ As $\tilde{R} \ne R_{\rho}$, we must have $|\beta_{\tilde{R}}| \ne |\alpha|$ and so 
	\[ |c_\rho| \ne \frac{(1-r^2) |\beta_{\tilde{R}}|}{1-r^2 |\beta_{\tilde{R}}|^2}.\]
	This gives the contradiction, because a straightforward computation shows that, using the definition of $\tilde{R}$, $|c_{\rho}|$ equals the above quantity. Thus, $R_\rho = \tfrac{R_1-R_2}{2} (1-\cos t)$, as needed. 	\end{proof}

We turn to the proof of Theorem \ref{thm:bdy2}.  In particular, we apply Theorem \ref{thm:circle} and use Remark \ref{rem:formulas} to obtain the points that the limiting-position envelope can contribute to the boundary.

\begin{proof}[Proof of Theorem \ref{thm:bdy2}.] Set $\Omega = \cup_{\alpha\in H(1, k)} D_\rho(\alpha,r).$ By Theorem \ref{thm:euclidean}, this set is also the union of the open Euclidean disks with boundaries $\mathcal{S}_t$, for $t\in[0,2\pi]$.
To see that $\{\mathcal{S}_{t}\}_{t \in[0, 2\pi]}$ satisfies the conditions of Theorem \ref{thm:circle}, observe that
\[ 
x_c(t) = c+a\cos t, \ \ \ y_c(t) = b\sin t,  \ \ \ r(t) = \tfrac{R_1-R_2}{2} (1-\cos t).\]
These functions are in $C^2([0,2\pi])$ and $r(t) >0$ for $t\in (0, 2\pi).$ Furthermore,
\[ x_c'(t)^2 +y_c'(t)^2 - r'(t)^2 = \frac{k^2(1-r^2)}{(1+r+k(1-r))(1-r+k(1+r))} = R_1 R_2>0,\]
so the derivative condition is satisfied. Theorem \ref{thm:circle} implies that 
\[ \partial \Omega \subseteq E_2 \cup \mathcal{S}_0 \cup \mathcal{S}_{2\pi}.\]
The formulas in Remark \ref{rem:formulas} show that the portion of $E_2 \setminus \{\mathcal{S}_0 \cup \mathcal{S}_{2\pi}\}$ that can contribute to $\partial \Omega$ is comprised of the following two curves $(x_1, y_1)$, $(x_2, y_2)$ parameterized by $t$ and 
defined as follows:
\[
\begin{aligned}
x_j(t) &= c +a\cos t +\frac{(R_1-R_2)(1-\cos t )}{2(a^2\sin^2t+ b^2\cos^2t)}\left(a \tfrac{R_1-R_2}{2}  \sin^2t   +(-1)^{j+1} b^2\cos t \right),\\
 y_j(t) &= b \sin t  +\frac{(R_1-R_2)(1-\cos t )}{2(a^2\sin^2t+ b^2\cos^2t)}\left( -b\tfrac{R_1-R_2}{2} \cos t \sin t  +(-1)^{j+1} ab \sin t \right), 
\end{aligned}
\]
 for $t\in (0, 2\pi).$ 
These simplify to
\[ 
\begin{aligned}
x_1(t) &= \frac{1+r^2+(k+2r-kr^2)\cos t}{1+k+r^2-kr^2+2r\cos t} \ \  \text{ and } \ \ y_1(t) = \frac{(1+r+k-kr)(1+r)}{1+k+r^2-kr^2+2r\cos t} b\sin t,\\
x_2(t) &=\frac{1+r^2+(k-2r-kr^2)\cos t}{1+k+r^2-kr^2-2r\cos t} \ \  \text{ and } \ \  y_2(t) = \frac{(1-r+ k+kr)(1-r)}{1+r^2+k-kr^2-2r\cos t }b\sin t. 
\end{aligned}
\]
The denominators never vanish because
\[ \left(1+k+r^2-kr^2\pm 2r\cos t\right) \ge \left(1+k+r^2-kr^2-2r\right) = (1-r)^2+k(1-r^2) >0,\]
and so $x_j(t), y_j(t)$ are well defined for all $t$. 

We claim that the image of each $(x_j, y_j)$ equals $\partial \mathscr{D}_j.$ First, a straightforward computation shows that each $(x_j(t), y_j(t)) \in \partial \mathscr{D}_j$. For the reverse containment, consider $(x_1(t), y_1(t))$. Observe that $t=0, 2\pi$ gives the point $(1,0)$ and $t=\pi$ gives the other point on $\partial \mathscr{D}_1 \cap \mathbb{R}$.  
When $t \in (0, \pi)$, we can see that $y_1(t)>0$ and when $t \in (\pi, 2\pi)$, $y_1(t) <0$.  By continuity, this implies that $(x_1(t), y_1(t))$ must map out the entire circle $\partial \mathscr{D}_1$. A similar argument shows that  $(x_2(t), y_2(t))$ maps out $\partial \mathscr{D}_2$. To identify $E_2$ with these formulas, we need to restrict to $t \in (0,2\pi)$, which only omits the point $(1,0)$.

Then Theorem \ref{thm:circle} implies that
\[\partial \Omega \subseteq \left( (\partial\mathscr{D}_1 \cup \partial \mathscr{D}_2) \setminus \{(1,0)\} \right) \cup \mathcal{S}_0 \cup \mathcal{S}_{2\pi} = \partial \mathscr{D}_1 \cup \partial \mathscr{D}_2,\]
as $\mathcal{S}_0 = \mathcal{S}_{2\pi} = (1,0).$

Now we need to show that $\partial \mathscr{D}_1 \cup \partial \mathscr{D}_2 \subseteq \partial \Omega.$ To that end, recall that $\mathscr{D}_2 \subseteq  \mathscr{D}_1.$ We first show
$\partial \mathscr{D}_1 \subseteq \partial \Omega.$ By way of contradiction, assume that there is some $p \in \partial \mathscr{D}_1 \cap \partial \Omega^c.$ As $p \in \mathcal{S}_t$ for some $t$,
this implies that there is some open disk $B_{\epsilon}(p)$ centered at $p$ that is contained in $\Omega$. Let $(v_1, v_2)$ denote a vector normal to $\partial \mathscr{D}_1$ at $p$ pointing out of $\mathscr{D}_1$. 
Let $\tilde{s} = \sup\{ s \in \mathbb{R}: p+ s(v_1, v_2) \in \Omega\}.$ Then $\tilde{p} =  p+ \tilde{s}(v_1, v_2) \in \partial \Omega$. However, as $\tilde{s} \ge \epsilon$, it is easy to see that $\tilde{p} \not \in \partial \mathscr{D}_1 \cup \partial \mathscr{D}_2$, a contradiction.

To show that $\partial \mathscr{D}_2 \subseteq \partial \Omega,$ first observe that for $t\in [0, 2\pi]$
\[ \begin{aligned}
(x_c(t) - c_1)^2 + y_c(t)^2 &= (R_1-r(t))^2, \\
(x_c(t) - c_2)^2 + y_c(t)^2 &= (R_2+r(t))^2,
\end{aligned}
\] 
so Remark \ref{rem:intersection} implies that each $\mathcal{S}_t$ is tangent to both $\partial \mathscr{D}_1$ and $\partial \mathscr{D}_2$. The one common point between $\partial \mathscr{D}_1$ and $\partial \mathscr{D}_2$ is $(1,0)$ and for $t\in (0, 2\pi)$, $(1,0)$ is 
 not in the disk with boundary $\mathcal{S}_t$. As $\mathscr{D}_2 \subseteq \mathscr{D}_1$, this means $\mathcal{S}_t$ must contain a point outside of $\overline{\mathscr{D}_2}.$ Thus $\mathcal{S}_t$ and $\partial \mathscr{D}_2$ must be externally tangent to each other for $t\in (0, 2\pi)$ and so, $\mathscr{D}_2 \cap \Omega = \emptyset.$ Given that, by way of contradiction, assume that there is some $p \in \partial \mathscr{D}_2 \cap \partial \Omega^c.$ As $p \in \mathcal{S}_t$ for some $t$,  there must be some open disk $B_{\epsilon}(p)$ centered at $p$ that is contained in $\Omega.$ As $B_{\epsilon}(p) \cap \mathscr{D}_2 \ne \emptyset,$ this gives the contradiction. 
\end{proof}

In this situation, we showed that the center of each $\mathcal{S}_t$ lies on the ellipse given in \eqref{eqn:E} with center $(c_1+c_2)/2$ and major axis $R_1 +R_2$. Then we deduced that $\mathcal{S}_t$ is tangent to both $\partial \mathscr{D}_1$ and $\partial \mathscr{D}_2.$ As discussed in the following remark, the converse is also true.

\begin{remark} Let $\mathscr{C}_1$ and $\mathscr{C}_2$ be circles with centers $d_1$ and $d_2$ and radii $r_1$ and $r_2.$ Assume that $\mathscr{C}_2$ is internally tangent to $\mathscr{C}_1.$ Let $\mathcal{S}$ be a circle with center $c$ and radius $r$ that is tangent to $\mathscr{C}_1$ at a point $\beta_1$ and $\mathscr{C}_2$ at a point $\beta_2$. 

\begin{figure}[H]
	\centering
	\includegraphics[width=6cm]{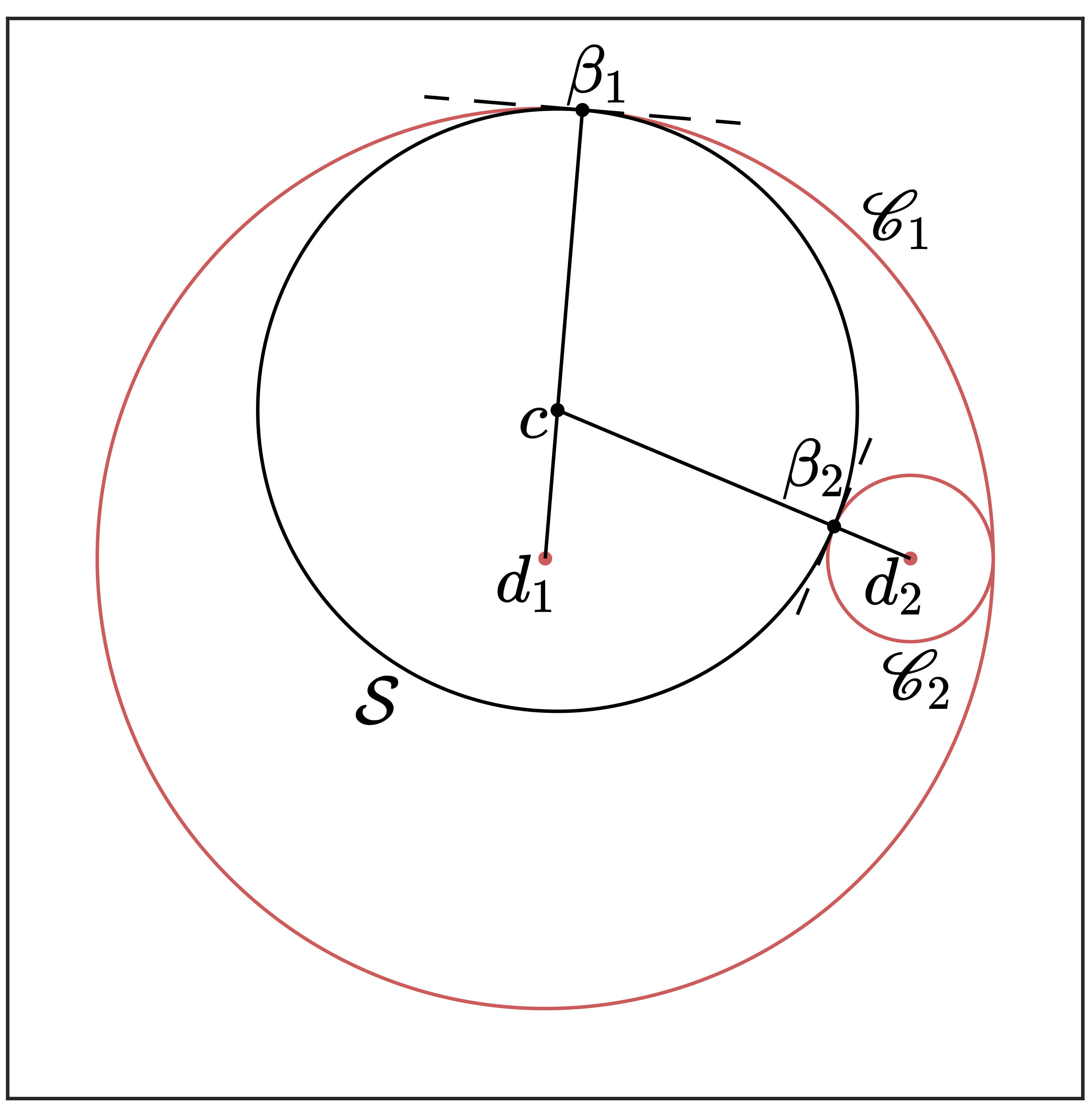}
\caption{\small Two tangent circles $\mathscr{C}_1$ and $\mathscr{C}_2$ with a mutually tangent circle $\mathcal{S}$. }
\end{figure}

We claim that $c$ must lie on the ellipse with center $(d_1 +d_2)/2$ and major axis $r_1 +r_2.$ 
To see this, observe that $\mathcal{S}$ must be internally tangent to $\mathscr{C}_1$ and externally tangent to $\mathscr{C}_2$. Moreover, the line from $\beta_1$ to $d_1$ must go through $c$ and the line from $c$ to $d_2$ must go through $\beta_2$. This gives
\[ |c - d_1| + |c - d_2 | = (r_1-r) + (r_2 +r) = r_1 +r_2.\]
 This means $c$ is on the ellipse with center $(d_1+d_2)/2$ and major axis $r_1+r_2,$ as needed.
\end{remark}

Finally, we observe that Theorem \ref{thm:bdy2} also has applications to interpolation:
\begin{remark} \label{rem:app2} Theorem \ref{thm:bdy2} is also of interest in the context of interpolation.  In \cite{W}, Wortman shows that given a convex curve $\Gamma$ in the closed unit disk with $\Gamma \cap \mathbb{T} = \{\lambda\}$ and $\Gamma$ tangent to $\mathbb{T}$ at $\lambda$, and a sequence $(z_n)$ on $\Gamma$ satisfying $\rho(z_n, z_{n+1}) = \delta > 0$ for all $n$, then $(z_n)$ is an interpolating sequence. Thus, choosing separated points along the boundary of the envelope $\partial \mathscr{D}_1 \cup \partial \mathscr{D}_2$ yields an interpolating sequence. We note that these results were generalized to a larger set of curves (called $K$-curves) by Max L. Weiss, \cite{W75}. Weiss also considers $M$-sequences; that is, sequences of the form $(z_n)$ where $z_n = r_n e^{i \theta_n}$ tends to $1$, $r_n$ is a strictly increasing sequence converging to $1$, $\theta_n$ is a strictly increasing sequence converging to $0$, and $\theta_n/(1-r_n) \to \infty$ as $n \to \infty$.  These curves also played an important role in the study of the maximal ideal space of $H^\infty$, \cite{B90}. In particular, the study of the envelopes of pseudohyperbolic disks is intimately connected with the study of interpolation.
\end{remark}


\begin{thebibliography}{12}

\bibitem{B90} 
P. Budde. Support sets and Gleason parts. \emph{Michigan Math. J.} \textbf{37} (1990), no. 3, 367--383.

\bibitem{bruce}
J.W. Bruce and P.J. Giblin.
What is an envelope? 
\emph{Math. Gaz}. \textbf{65} (1981), no. 433, 186--192. 

\bibitem{cap02} G. Capitanio.
On the envelope of 1-parameter families of curves tangent to a semicubic cusp. 
\emph{C. R. Math. Acad. Sci. Paris} \textbf{335} (2002), no. 3, 249--254. 

\bibitem{c88}
R. Courant. 
\newblock
\emph{Differential and integral calculus}. Vol. II. Translated from the German by E. J. McShane. Reprint of the 1936 original. Wiley Classics Library. A Wiley-Interscience Publication, John Wiley \& Sons, Inc., New York, 1988.

\bibitem{d71}
C. Davis.
The Toeplitz-Hausdorff theorem explained. 
\emph{Canad. Math. Bull.} \textbf{14} (1971), 245--246.

\bibitem{don57}
W. F. Donoghue Jr.
On the numerical range of a bounded operator. 
\emph{Michigan Math. J.} \textbf{4} (1957) 261--263. 

\bibitem{Gar07} J.B. Garnett, Bounded analytic functions. Revised first edition. Graduate Texts in Mathematics, \textbf{236}. Springer, New York, 2007.

\bibitem{gil12} T. Gilbert.
Lost mail: the missing envelope in the problem of the minimal surface of revolution. 
\emph{Amer. Math. Monthly} \textbf{119} (2012), no. 5, 359--372. 

\bibitem{GPV07} D. Girela,  J.A. Pel\'aez,and D. Vukoti\'c.
Integrability of the derivative of a Blaschke product. 
\emph{Proc. Edinb. Math. Soc. (2)} \textbf{50} (2007), no. 3, 673--687.

\bibitem{GPV08} D. Girela,  J.A. Pel\'aez,and D. Vukoti\'c. Interpolating Blaschke products: Stolz and tangential approach regions. \emph{Constr. Approx.} \textbf{27} (2008), no. 2, 203--216. 


\bibitem{gr97} K.E. Gustafson and  D.K.M.~Rao, \emph{Numerical range. The field of values of linear operators and matrices.} Universitext. Springer-Verlag, New York, 1997.

\bibitem{hj91} R.A. Horn and C.R. Johnson, Topics in Matrix Analysis, Cambridge University Press, 1991.

\bibitem{hj90} R.A. Horn and C.R. Johnson,  Matrix Analysis, Cambridge University Press, 1990.

\bibitem{Ho} K. Hoffman. \emph{Banach spaces of analytic functions.} Prentice-Hall Series in Modern Analysis Prentice-Hall, Inc., Englewood Cliffs, N. J. 1962.


\bibitem{HW16} M. Huibregtse and A. Winchell. Envelope curves and equidistant sets.  \emph{Involve} \textbf{9} (2016), no. 5, 839--856. 

\bibitem{km96}
P. Kahlig and J. Matkowski.
Envelopes of special class of one-parameter families of curves. 
\emph{Demonstratio Math.} \textbf{29} (1996), no. 4, 799--806. 

\bibitem{k07} 
D. Kalman.
\newblock
Solving the ladder problem on the back of an envelope.
\emph{Math. Mag.} \textbf{80} (2007), no. 3, 163--182. 

\bibitem{K08}
R. Kippenhahn. {On the numerical range of a matrix.} Translated from the German 
by Paul F. Zachlin and Michiel E. Hochstenbach. \emph{Linear Multilinear Algebra} \textbf{56} (2008), no. 1-2, 185--225.


\bibitem{ckli}
C.-K. Li,
\newblock
A simple proof of the elliptical range theorem.
\emph{Proc. Amer. Math. Soc.} \textbf{124} (1996), no. 7, 1985--1986.

\bibitem{M2010} F. Maclachlan, Long-Run and Short-Run Cost Curves,{\it Famous figures and diagrams in economics}. In Blaug, M., Lloyd, P.  eds., Edward Elgar Publishing, (2010).

\bibitem{ma13} V.J. Matsko. Generic ellipses as envelopes. \emph{Mathematics Magazine,}
\textbf{86}, (2013), no. 5, 358--365.



\bibitem{MR2015} R. Mortini, R. Rupp,  On a family of pseudohyperbolic disks. 
\emph{Elem. Math.} \textbf{70} (2015), no. 4, 153--160.

\bibitem{fdm32}
F.D. Murnaghan, On the field of values of a square matrix, \textbf{18}, \emph{Proc. Natl. Acad. Sci.} (1932), no. 3,
246--248.

\bibitem{N59} D.J. Newman. Interpolation in $H^{\infty}$. \emph{Trans. Amer. Math. Soc.} \textbf{92} (1959) 501--507. 



\bibitem{N12} J. No\"el, Structures alg\'ebriques dans les anneaux fonctionnels, Th\'ese de doctorat,
LMAM, Universit\'e de Lorraine, Metz, 2012.
http://docnum.univ-lorraine.fr/public/DDOC T 2012 0222 NOEL.pdf


\bibitem{p47} L.~A. Pars. 
A note on the envelope of a certain family of curves. 
\emph{J. London Math. Soc.} \textbf{22}, (1947). 25--31. 

\bibitem{PP00} H. Pottman, M. Peternell, Envelopes---computational theory and applications, \emph{Proceedings
of Spring Conference in Computer Graphics}, (2000) Budmerice, Slovakia, pp. 3--23.

\bibitem{tak17} M. Takahashi.
Envelopes of Legendre curves in the unit tangent bundle over the Euclidean plane.
\emph{Results Math.} \textbf{71} (2017), no. 3-4, 1473--1489. 

\bibitem{rutter} J.W.~Rutter.
\emph{Geometry of curves.} 
Chapman \& Hall/CRC Mathematics. Chapman \& Hall/CRC, Boca Raton, FL, 2000. 

\bibitem{W75} M.L. Weiss.
Some $H^{\infty}$-interpolating sequences and the behavior of certain of their Blaschke products. 
\emph{Trans. Amer. Math. Soc.} \textbf{209} (1975), 211--223. 

\bibitem{W} D.H. Wortman. Interpolating sequences on convex curves in the open unit disc. \emph{Proc. Amer. Math. Soc.} \textbf{48} (1975), 157--164.


\end{thebibliography}
\end{document}